\documentclass[reqno,a4paper]{amsart}

\usepackage[latin1]{inputenc}
\usepackage[english]{babel}
\usepackage{eucal,amsfonts,amssymb,amsmath,amsthm,epsfig,mathrsfs}
\usepackage{cancel,soul}
\usepackage{color}
\usepackage{amscd,amsxtra}
\usepackage{enumerate}
\usepackage{latexsym}
\usepackage{bm}

\allowdisplaybreaks

\newcounter{ipotesi}

\Alph{ipotesi}
\makeatletter
\@namedef{subjclassname@2020}{%
  \textup{2020} Mathematics Subject Classification}
\makeatother
 \makeatletter \@addtoreset{equation}{section}

\makeatother \makeatletter

\newtheorem{thm}{Theorem}[section]
\newtheorem{hyp}[thm]{Hypotheses}{\rm}
\newtheorem{hyp0}[thm]{Hypothesis}{\rm}
\newtheorem{lemm}[thm]{Lemma}
\newtheorem{cor}[thm]{Corollary}
\newtheorem{prop}[thm]{Proposition}

\newtheorem{rmk}[thm]{Remark}{\rm}

\newcounter{parentenv}

\newcommand{\R}{{\mathbb R}}

\newcommand{\N}{{\mathbb N}}

\newcommand{\Rd}{\mathbb R^d}

\newcommand{\eps}{\varepsilon}
\newcommand{\ra}{\rightarrow}

\newcommand{\lip}{\textrm{Lip}}

\renewcommand{\hat}[1]{\widehat{#1}}
\renewcommand{\tilde}[1]{\widetilde{#1}}

\newcommand{\set}[1]{{\left\{#1\right\}}}
\newcommand{\pa}[1]{{\left(#1\right)}}
\newcommand{\sq}[1]{{\left[#1\right]}}
\newcommand{\gen}[1]{{\left\langle #1\right\rangle}}
\newcommand{\abs}[1]{{\left|#1\right|}}
\newcommand{\norm}[1]{{\left\|#1\right\|}}

\newcommand{\eqsys}[1]{{\left\{\begin{array}{ll}#1\end{array}\right.}}
\newcommand{\tc}{\, \middle |\,}

\begin{document}

\frenchspacing

\title[Functional inequalities for generalised Mehler semigroups]{Functional inequalities \\
for some generalised Mehler semigroups}
\author[L. Angiuli, S. Ferrari and D. Pallara]{Luciana Angiuli, Simone Ferrari, Diego Pallara}
\address{L.A., S.F.:  Dipartimento di Matematica e Fisica ``Ennio De Giorgi'', Universit\`a del Salento, Via per Arnesano, I-73100 LECCE, Italy}
\address{D.P.:  Dipartimento di Matematica e Fisica ``Ennio De Giorgi'', Universit\`a del Salento,
and I.N.F.N., Sezione di Lecce, Via per Arnesano, I-73100 LECCE, Italy}
\email{luciana.angiuli@unisalento.it}
\email{simone.ferrari1@unipr.it}
\email{diego.pallara@unisalento.it}

\date{\today}

\begin{abstract}
We consider generalised Mehler semigroups and, assuming the existence of an associated invariant measure 
$\sigma$, we prove functional integral inequalities with respect to $\sigma$, such as logarithmic
Sobolev and Poincar\'{e} type. Consequently, some integrability properties of exponential functions with 
respect to $\sigma$ are deduced. 
\end{abstract}

\date{\today}
\maketitle

\section{Introduction}

Generalised Mehler semigroups are defined for real-valued, bounded and Borel measurable functions $f:X\ra\R$, i.e. $f\in B_b(X)$, by the formula
\begin{equation}\label{GMsgrp}
(P_tf)(x)=\int_X f(T_t x+y) \mu_t(dy),
\end{equation}
where $X$ is a (finite or infinite dimensional) Banach space, $(T_t)_{t\geq 0}$ is a strongly
continuous semigroup of bounded operators on $X$ and $(\mu_t)_{t\geq 0}$
is a family of Borel probability measures on $X$ verifying $\mu_0=\delta_0$ and
$\mu_{t+s}=(\mu_t \circ T^{-1}_s) * \mu_s$ for any $s,t \geq 0$.
The semigroup \eqref{GMsgrp} is related to the stochastic differential equation
\begin{equation}\label{Mehler_SDE}
\left\{\begin{array}{ll}
dZ(t) = AZ(t)dt + dY(t), &t>0;
\\
Z(0)=x\in X;
\end{array}\right.
\end{equation}
where $A:D(A)\subseteq X\to X$ is the infinitesimal generator of $T_t$ and $Y(t)$ is a L\'{e}vy process
in $X$, i.e., a stochastic process with c\`{a}dl\`{a}g trajectories starting at $0$ and having
stationary and independent increments. For $\xi\in X^*$, $t>0$ we have
$\mathbb{E}[i\xi Y(t)]=\exp(-t\lambda(\xi))$ and $\mu_t$ is defined through its characteristic function
$\hat{\mu}_t(\xi)=\exp(-\int_0^t\lambda(T_s^*\xi)ds)$, see \cite{Bog3}.
By the L\'{e}vy-Khinchine theorem, the function $\lambda$ is determined by its
{\em characteristics} $[b,Q,M]$ with $b\in X$, $Q$ is a nonnegative definite symmetric
trace-class operator on $X$
and $M$ is a L\'{e}vy measure, see \eqref{lambda} below. The semigroup $P_t$ is related to
\eqref{Mehler_SDE} by
\[
P_tf(x)=\mathbb{E}[f(Z(t,x))],\qquad t\geq 0,\ x\in X, \, f \in B_b(X);
\]
where $Z(t,x)$ is the (mild or weak) solution of \eqref{Mehler_SDE}.
If $Y(t)=\sqrt{Q}W(t)$, where $W(t)$ is a Brownian motion (i.e. $M\equiv 0$), then, setting
$Q_tx=\int_0^tT_sQT_s^*xds$, $\mu_t=\mathcal{N}(0,Q_t)$ and $P_t$ is the
Ornstein--Uhlenbeck semigroup given by the classical Mehler formula. In this case the trajectories are
continuous, whereas in the general case $Y(t)$ may have jumps giving rise to nonlocal effects. Indeed,
the (weak) generator of the semigroup $P_t$ is in general a nonlocal, or
pseudodifferential operator (see \cite{LR02} and the example in Subsection \ref{OUexample}) and is given by
\begin{equation}\label{generator}
\mathcal{L}f(x)=\frac{1}{2} {\rm Tr}[QD^2f](x) + \langle x, A^*D f(x)\rangle
+\int_X[f(x+y)-f(x)-\gen{Df(x),y}\chi_{B_1}(y)]M(dy),
\end{equation}
on regular functions.
We refer to Section \ref{SecPrelim} for a more detailed explanation, to \cite{PZ07} for a general
introduction to these topics, to 
\cite{App07,App15,Bog3,Chojnowska,FR,LR02,LR04,PriTra16,PriZab06,RW03,WangBook} 
as more specific basic references to generalised Mehler semigroups and to the very recent \cite{LR} and the 
reference therein for an updated account on the regularity theory, which we do not discuss here.

In this paper we always assume that $X$ is a separable Hilbert space and that there exists a unique 
invariant measure $\sigma$ associated to $P_t$, keeping the conditions given in \cite{FR}, and 
look for functional inequalities with respect to $\sigma$. The most classical ones are the {\em logarithmic
Sobolev inequalities} coming back to \cite{Gro75,Gro93} and \cite{Federbush}, a theory widely
developed in the Wiener case $M\equiv 0$. We refer to \cite{Ane,LedAMS,Royer,WangBook} and the 
reference therein, as well as to \cite{AFP19,BF20,Cap19,DPDG02,Kaw06} for more recent results. 
For the general case little is known and we refer to \cite{Chafai} for a general discussion of 
functional inequalities related to entropy.
In the general case of processes with jumps such estimates are not available, as pointed out e.g.
in \cite{BL98,RW03,WangBook,Wu00}. Therefore, as done in the quoted papers, we study modifications of such estimates.
In particular, we estimate the {\em entropy} of positive measurable functions $f$ by the integral of some relative increments of $f$ with respect
to the L\'{e}vy measure $M$, which is charged to take into account the nonlocal effects. Accordingly,
our estimates hold true for positive functions whose infimum is far from $0$, see Theorem \ref{thm_neve}.
From these modified logarithmic Sobolev type inequalities we derive Poincar\'e inequalities on a
suitable class of functions and we study the exponential integrability of Lipschitz continuous functions, which in our
framework appears to be the natural counterpart of the classical Fernique theorem in a Gaussian
context. As further consequences of the basic estimates, comparisons of moments of the
measures $M$ and $\sigma$ are provided.

In order to simplify the presentation, we have performed all the computations assuming that
$Q=0$ in the above recalled representation of the function $\lambda$, which amounts to
saying that there is no diffusion term in the generator, see \eqref{generator} and \eqref{formula_L}.
We stress that this is not restrictive, because the general case can be recovered by standard arguments.
Indeed, at the end of Sections \ref{sec_log}, \ref{sec_exp}, \ref{SecEx} we discuss the
adaptation of the proofs and the results presented in each section needed to extend them
to the general case. In particular, in Remark \ref{gaussiana} we describe the new invariant measure
and how the entropy estimates must be modified, in Remark \ref{gauss_exp} we sketch how
the statement and the proof of Theorem \ref{promosso} must be modified to get exponential
integrability with respect to the new invariant measure and in Remark \ref{gauss_ex}
we point out that also the examples can easily be generalised to the general case.

The paper is organised as follows. In Section \ref{SecPrelim} we recall the notation
we use and collect the main results on generalised Mehler semigroups concerning the weak generator,
the measures $\mu_t$ and the exponential function $\lambda$ that appears in connection to te L\'{e}vy
process $Y$. In Section \ref{SecInvMeas} we recall a condition ensuring the existence of an invariant
measure $\sigma$ for $P_t$, extend the semigroup to the $L^p(X,\sigma)$ spaces and describe its
asymptotic behaviour. In Section \ref{sec_log} we prove the main logarithmic Sobolev type integral
inequalities, in Section \ref{sec_exp} we study the exponential integrability of Lipschitz continuous
functions. In particular, we deduce from the estimates in Section \ref{sec_log} an estimate on
the size of the tail of the distribution of Lipschitz continuous functions.
Finally, in Section \ref{SecEx} some examples of semigroups to which our results apply are presented.

\paragraph*{\bf Acknowledgements}
S.F. has been partially supported by the OK-INSAID project ARS01-00917.
The authors are members of G.N.A.M.P.A. of the Italian Istituto Nazionale di Alta Matematica (INdAM)
and have been partially supported by the PRIN 2015 MIUR project 2015233N54. \\
The authors are grateful to Alessandra Lunardi and Enrico Priola for many helpful conversations. 

\section{Notation and Preliminaries}\label{SecPrelim}

For any $a,b \in \R$ we set $a\vee b:=\max\{a,b\}$ and $a\wedge b:=\min\{a,b\}$.
Let $X$ be a real separable Hilbert space, that can be either finite or infinie dimensional, with
inner product $\langle\cdot,\cdot\rangle_X$ and associated
norm $\abs{\cdot}_X$, and let $X^*$ be its topological dual. When there is no risk of confusion we drop the
$X$ from the symbols. $\mathcal{B}(X)$ denotes the Borel $\sigma$-algebra of $X$ and $B_b(X)$
the space of real-valued bounded Borel functions on $X$. We denote $B_1$ the open unit ball centred at
the origin in $X$. If $f:X\ra\R$ is a Fr\'{e}chet differentiable function, we denote by $Df$ its
Fr\'{e}chet derivative.

The symbol $\mathcal{L}(X)$ denotes the space of bounded linear operators from $X$ to itself and
$I$ denotes the identity operator. The domain of a linear operator $A$ on $X$ is denoted $D(A)$
and its range ${\rm Ran}(A)$. An operator $T\in\mathcal{L}(X)$ is Hilbert--Schmidt if
\[
\sum^{\infty}_{n=1}\abs{Te_n}^2<\infty,
\]
for some (hence all) orthonormal basis $\{e_n\,|\, n\in\N\}$ of $X$. An operator $T\in \mathcal{L}(X)$ is
trace-class if it is compact and the series $\sum_k|\lambda_k|$ of its eigenvalues $(\lambda_k)_{k\in\N}$,
counted with their multiplicity, is convergent.

The \emph{Sazonov topology on $X$} is the topology generated by the family of seminorms $x \mapsto |Tx|$,
where $T$ ranges over all Hilbert--Schmidt operators on $X$ and it plays an important role in the
definition of L\'{e}vy processes and generalised Mehler semigroups. We refer to
\cite{BF75,Part67Pro} for an in-depth study of all this notions.

If $\mu$ and $\gamma$ are two finite Borel measures on $X$, we denote by $\hat{\mu}$ the characteristic
function of $\mu$ and by $\mu * \gamma$ the convolution measure defined by
\begin{align*}
\hat{\mu}(\xi)&:=\int_Xe^{i\langle x,\xi\rangle}\mu(dx), \qquad\xi\in X^*;
\\
[\mu * \gamma] (E) &:= \int_X  \mu(E-x)\gamma(dx), \qquad E\in \mathcal{B}(X).
\end{align*}
If $T \in \mathcal{L}(X)$ we denote by $\mu\circ T^{-1}$ the image measure defined as
$(\mu \circ T^{-1})(B):=\mu(T^{-1}(B))$ for any $B \in \mathcal{B}(X)$.

A \emph{generalised Mehler semigroup} on $X$ is defined by the formula
\begin{equation}\label{mehler}
(P_tf)(x)=\int_X f(T_t x+y) \mu_t(dy), \qquad f\in B_b(X),
\end{equation}
where $(T_t)_{t\geq 0}$ is a strongly continuous semigroup of linear operators on $X$ and $(\mu_t)_{t\geq 0}$
is a family of Borel probability measures on $X$. The semigroup law for $(P_t)_{t\geq 0}$ is equivalent
to the following property of the family $(\mu_t)_{t\geq 0}$:
\begin{align}\label{semigroup_prop}
\mu_0=\delta_0,\qquad \mu_{t+s}=(\mu_t \circ T^{-1}_s) * \mu_s\qquad  \text{ for all }s,t \geq 0,
\end{align}
see \cite[Proposition 2.2]{Bog3}. We recall, see \cite[Lemma 2.6]{Bog3}, that if for any $\xi\in X^*$
the function $t\mapsto\widehat{\mu_t}(\xi)$ is absolutely continuous on $[0,\infty)$ and differentiable
at $t=0$ then, setting
\[
\lambda(\xi):=-\frac{d}{dt}\widehat{\mu_t}(\xi)_{|_{t=0}},
\]
the function
$t\mapsto \lambda(T_t^*\xi)$ belongs to $L^1_{\rm loc}((0,\infty))$, hence \eqref{semigroup_prop}
is equivalent to
\begin{align}\label{hatmulambda}
\hat{\mu}_t(\xi)= \exp\left(-\int_0^t \lambda (T_s^* \xi)ds\right), \qquad t\ge 0,\ \xi\in X^* .
\end{align}
In this case $\lambda$ is {\em negative definite}, i.e., the matrices whose entries are
$(\lambda(\xi_i-\xi_j))_{i,j=1,\ldots,n}$ are negative definite for every $n\in \N$ and
for every $n$-tuple $(\xi_1,\ldots,\xi_n)\in X^*$. Throughout
the paper we assume that $\lambda$ is also Sazonov continuous on $X^*$. This implies that, for every
$t\geq 0$, the functions $e^{-t\lambda}$ are positive definite (see \cite{BF75}) and Sazonov continuous.
Therefore, by \cite[Theorem VI.1.1]{VTC87}, they are characteristic functions of probability measures
on $X$. This implies that $e^{-t\lambda}$ is the characteristic function of an infinitely divisible
probability measure on $X$. Using the L\'{e}vy--Khinchine theorem, (see \cite[Theorem VI.4.10]{Part67Pro}),
there are $b\in X$, a nonnegative self-adoint trace-class operator $Q\in {\mathcal L}(X)$ and a L\'{e}vy
measure $M$, that is a Borel measure satisfying
\begin{align}\label{LevyM}
M(\{0\})=0, \qquad \int_X (1 \wedge |x|^2)M(dx)<\infty,
\end{align}
such that $\lambda$ can be written in the form
\begin{equation}\label{lambda}
\lambda(\xi)=-i\langle \xi, b\rangle+\frac{1}{2}\langle Q\xi,\xi\rangle
-\int_X \left(e^{i\langle x, \xi\rangle}-1-i\langle x, \xi\rangle\chi_{B_1}(x)\right)M(dx).
\end{equation}
In the sequel we use the symbol $\leftrightarrow$ to associate the triple $[b,Q,M]$ with $\lambda,\mu_t$
and $\hat{\mu}_t$, according to \eqref{lambda}, \eqref{hatmulambda}.
It is immediate to check that $P_t$ maps $C_b(X)$ into itself and
$$
\|P_t f\|_\infty\le \|f\|_\infty, \qquad t>0,\, f \in C_b(X),
$$
but, in general, $P_t$ is not strongly continuous in $C_b(X)$. The continuity of the map
$(t,x)\mapsto P_tf(x)$, $f\in C_b(X)$ allows us to define the \emph{weak generator} $\mathcal{L}$
through its resolvent
$$
[R(\gamma, \mathcal{L})f](x)=\int_0^{\infty} e^{-\gamma t}P_tf(x)dt
$$
for any $\gamma>0$, $f\in C_b(X)$ and $x\in X$. Let $A:D(A)\subseteq X \to X$ be the infinitesimal
generator of the semigroup $(T_t)_{t\geq 0}$. We recall that by \cite[p. 40]{App15} if $Q=0$ we have
\begin{equation}\label{formula_L}
\mathcal{L}f(x)=\langle Ax, D f(x)\rangle+\int_X[f(x+y)-f(x)-\gen{Df(x),y}\chi_{B_1}(y)]M(dy),
\end{equation}
for any $f\in \mathcal{F}C^2_b(X)$. Note that for such functions, the integral in \eqref{formula_L} is
well defined by the Taylor formula.

In the sequel it will be useful to consider a core for the generator of $P_t$ in $C_b$ with respect to
the mixed topology $\tau_m$ on $C_b(X)$, i.e., the finest locally convex topology that agrees on norm
bounded sets with the topology of uniform convergence on compacts (see \cite{GK01} for a more in-depth
discussion about this topology). To do that we state the following hypothesis, see \cite{RW03}.
\begin{hyp0}\label{hyp_App}
There exists an orthonormal basis $\{h_n\,|\, n\in \N\}$ of $X$ consisting of eigenvectors of $A^*$ and
\begin{align}\label{App_cond}
\int_{B_1^c}|x|M(dx)<\infty.
\end{align}
\end{hyp0}
Following \cite{App07} (see also \cite[Remark 5.11]{PriTra16}), we say that $f\in C^2_A(X)$ if $f\in C_b(X)$
belongs to $C^2(X)$, its first and second order derivatives are uniformly bounded and uniformly continuous on
bounded subsets of $X$, ${\rm Ran}(Df)\subseteq D(A^*)$ and
$x\mapsto\langle x,A^*Df(x) \rangle\in C_b(X)$. We say $F\in \mathcal{F}C^2_A(X)$ if there exist
$n\in\N$ and $f\in C^2_b(\R^n)$ such that
\begin{align*}
F(x)=f(\gen{x,h_1},\ldots,\gen{x,h_n}),\qquad x\in X.
\end{align*}
In \cite[Theorem 5.2]{App07}, see also \cite[Remark 5.11]{PriTra16}, it is shown that $\mathcal{F}C^2_A(X)$,
under Hypothesis \ref{hyp_App}, is a core for the generator of $P_t$ in $C_b(X)$ equipped with the mixed topology. Recall that a {\em core} of an operator $A:D(A)\subseteq X\to X$ is a subspace $C\subseteq D(A)$ which is dense in $D(A)$ with respect to the graph norm $\norm{\cdot}_A:=\abs{\cdot}+\abs{A\cdot}$.
In the next section we use this result to prove that $\mathcal{F}C^2_A(X)$ is also a core for
$\mathcal{L}$ in $L^2(X,\sigma)$ as well, when an invariant measure $\sigma$ exists for the semigroup.

\section{Invariant measure}\label{SecInvMeas}

In this section we recall some conditions implying the existence of a unique invariant measure.
A Borel probability measure $\sigma$ on $X$ is an invariant measure for
$(P_t)_{t\geq 0}$ if
\begin{equation}\label{invariance}
\int_X P_t f d\sigma=\int_X fd\sigma,\qquad t\geq 0,\ f \in B_b(X);
\end{equation}
or, equivalently, $\sigma=(\sigma \circ T_t^{-1})*\mu_t$ for any $t>0$, where $\mu_t$ are the
measures in \eqref{mehler}. Throughout this section we consider $\lambda\leftrightarrow [b,0,M]$,
with $\hat{\mu}_t\leftrightarrow [b_t,0,M_t]$ according to \eqref{hatmulambda} and \eqref{lambda}, where
\begin{equation}\label{scrivere meglio}
b_t:=\int_0^t T_s b\, ds+\int_0^t\int_X T_s x\Big(\chi_{B_1}(T_s x)-\chi_{B_1}(x)\Big)M(dx)ds
\end{equation}
and $M_t$ are Borel measures defined setting $M_t(\{0\})=0$ and
\begin{equation}\label{defM_t}
M_t(B):=\int_0^t M(T_s^{-1}(B))ds, \qquad B \in \mathcal{B}(X), 0\notin B.
\end{equation}
Note that $M_t$ are L\'{e}vy measures. Indeed, as $T_t$ is strongly continuous, there exist $K\ge 1$ and
$\omega \in \R$ such that $|T_t x|\le Ke^{\omega t}|x|$ for any $t>0$ and $x \in X$. Hence
\begin{align}\label{est_M}
\int_X (1 \wedge |x|^2)M_t(dx)=&\int_0^t \int_X (1 \wedge |T_sx|^2)M(dx)ds
\notag\\
\le& \frac{K^2}{2\omega}(e^{2\omega t}-1) \int_X (1 \wedge |x|^2)M(dx)<\infty.
\end{align}
Following \cite[Theorem 3.1]{FR} we assume the following hypotheses that guarantee the existence and
the uniqueness of an invariant measure for $(P_t)_{t\geq 0}$.
\begin{hyp}\label{base}
Let $\lambda\leftrightarrow [b,0,M]$, let $(T_t)_{t\geq0}$ be a strongly continuous semigroup and
let $b_t,M_t$ be given in \eqref{scrivere meglio} and \eqref{defM_t}. Assume
\begin{enumerate}[\rm (i)]
\item \label{topo} there exists $b_\infty:=\lim_{t \to \infty}b_t$ in $X$;
\item \label{lino} setting $M_\infty:=\sup_{t>0} M_t$ (i.e., $M_\infty(\{0\})=0$ and
$M_\infty(B)=\int_0^\infty M(T_s^{-1}(B))ds$, $B\in \mathcal{B}(X)$, $0\notin B$), it holds that
\begin{equation*}
\int_0^\infty \int_X (1 \wedge |T_sx|^2)M(dx)ds<\infty;
\end{equation*}
\item $ \lim_{t \to \infty}T_t x=0$ in $X$ for every $x \in X$.
\end{enumerate}
\end{hyp}
The following result can be found in \cite[Section 3]{FR} and it is fundamental in most of the results of
this paper.

\begin{thm}\label{furhman}
Under Hypotheses \ref{hyp_App}, \ref{base} \eqref{topo} and \eqref{lino}, $M_\infty$  is a L\'{e}vy measure and the measure
$\sigma\leftrightarrow [b_\infty,0, M_\infty]$ is invariant for $P_t$. In addition, if Hypothesis
\ref{base}(iii) holds true, then $\sigma$ is unique and
\begin{equation}\label{inv_L}
\int_X \mathcal{L}f d\sigma=0
\end{equation}
for any $f \in \mathcal{F}C^2_A(X)$. Moreover $\mu_t$ converges weakly-star to $\sigma$ as $t \to \infty$.
\end{thm}

Let us show that $\mathcal{F}C^2_A(X)$ is a core for $\mathcal{L}$ in $L^2(X,\sigma)$.

\begin{lemm}
If Hypotheses \ref{hyp_App} and \ref{base} hold true, then $\mathcal{F}C^2_A(X)$ is invariant for $P_t$ and it is a core
for $\mathcal{L}$ in $L^2(X,\sigma)$.
\end{lemm}
\begin{proof}
We point out that $\mathcal{F}C^2_A(X)$ is invariant with respect to $P_t$ (see \cite[Theorem 5.2]{App07}
and \cite[Remark 5.11]{PriTra16}),
so to conclude we just need to show that $\mathcal{F}C^2_A(X)$ is contained in the domain of the generator
$\mathcal{L}$ in $L^2(X,\sigma)$ and that it is dense in $L^2(X,\sigma)$.

By \cite[Theorem 5.1(2)]{App07} the space $\mathcal{F}C^2_A(X)$
is contained in $D(\mathcal{L}_m)$, the domain of the generator of $P_t$ in $C_b(X)$ equipped with
the mixed topology. This means that for any $F\in\mathcal{F}C^2_A(X)$ there exists $G\in C_b(X)$ such that
\begin{align*}
\tau_m\text{-}\lim_{t\ra 0}\frac{P_tF-F}{t}=G.
\end{align*}
Let $\{t_n\}_{n\in\N}\in(0,\infty)$ be a sequence converging to zero. By \cite[Proposition 2.3]{GK01} the
sequence $((1/t_n)(P_{t_n}F-F)-G)_{n\in\N}$ is uniformly convergent to zero on compact subsets of
$X$ and
\begin{align*}
\sup_{n\in\N}\norm{\frac{P_{t_n}F-F}{t_n}-G}_\infty<\infty.
\end{align*}
By the dominated convergence theorem we get that the sequence $((1/t_n)(P_{t_n}F-F)-G)_{n\in\N}$
converges to zero in $L^2(X,\sigma)$. Since the argument is independent on the choice of the sequence
$(t_n)_{n\in\N}$ we obtain
\begin{align*}
\lim_{t\ra 0}\norm{\frac{P_tF-F}{t}-G}_{L^2(X,\sigma)},
\end{align*}
hence $F$ belongs to the domain of $\mathcal{L}$ in $L^2(X,\sigma)$ and $\mathcal{L}F=\mathcal{L}_mF$.

The fact that $\mathcal{F}C^2_A(X)$ is dense in $L^2(X,\sigma)$ can be proved by using
that $\mathcal{F}C^2_A(X)$ is $\tau_m$-sequentially dense in $C_b(X)$ (see \cite[Lemma 2.6]{GK01}) and
the same arguments as above.
\end{proof}

The following equality will be useful later on.

\begin{lemm}
Assume that Hypotheses $\ref{hyp_App}$ and $\ref{base}$ hold true. For every $f\in \mathcal{F}C^2_A(X)$ and
every $\Phi\in C^2(\R)$ we have
\begin{align}\label{form_eu}
\int_X\Phi'(f)\cdot\mathcal{L}f d\sigma&=\int_X\int_X \Big[\Phi(f(x))-\Phi(f(x+y))+\Phi'(f(x))\Big(f(x+y)-f(x)\Big)\Big]M(dy)\sigma(dx).
\end{align}
\end{lemm}
\begin{proof}
By the invariance relation \eqref{inv_L} it suffices to prove that
\begin{align}\label{id_com}
\mathcal{L}(\Phi\circ f)=(\Phi'\circ f)(\mathcal{L}f)+\int_X
\left[(\Phi\circ f)(\cdot+y)-(\Phi\circ f)-(\Phi'\circ f)\Big(f(\cdot+y)-f\Big)\right]M(dy)
\end{align}
and to observe that $\Phi \circ f$ belongs to $\mathcal{F}C^2_A(X)$.
Formula \eqref{id_com} easily follows from \eqref{formula_L}. Indeed, we have
\begin{align*}
[\mathcal{L}(\Phi\circ f)](x) &=[(\Phi'\circ f)(x)]\langle Ax,Df(x)\rangle
\\
&+\int_X\Big[(\Phi\circ f))(x+y)-(\Phi\circ f))(x)-[(\Phi'\circ f)(x)]
\gen{D f(x),y}\chi_{B_1}(y)\Big]M(dy).
\end{align*}
Now adding and subtracting $\int_X [(\Phi'\circ f)(x)][f(x+y)-f(x)]M(dy)$ we get \eqref{id_com}.
\end{proof}

In the following proposition we collect the main properties of the semigroup $P_t$ in the space
$L^p(X,\sigma)$, $p\in [1,\infty)$.

\begin{prop}
Assume that Hypotheses $\ref{hyp_App}$ and $\ref{base}$ hold true. The semigroup $P_t$ can be extended to a contractive strongly continuous semigroup (still denoted by $P_t$)
on $L^p(X,\sigma)$ for any $1\le p<\infty$.
\end{prop}
\begin{proof}
The Jensen inequality, formula \eqref{mehler} and the invariance property \eqref{invariance} yield that,
for any $f \in C_b(X)$
\begin{align*}
\int_X |P_t f|^p d\sigma \le \int_X P_t|f|^pd\sigma = \int_X |f|^pd\sigma
\end{align*}
whence $\|P_t f\|_{L^p(X,\sigma)}\le \|f\|_{L^p(X,\sigma)}$ for any $f \in C_b(X)$. Moreover, since the
measure $\sigma$ is a probability Borel measure, the space $C_b(X)$ is dense in $L^p(X,\sigma)$ for any
$p\in [1,\infty)$ (see Lemma \ref{lusin BUC}). Thus we can extend $P_t$ to a bounded
linear operator in $L^p(X,\sigma)$ with $\|P_t\|_{\mathcal{L}(L^p(X,\sigma))}\le 1$.
Now, let us prove that $P_t$ is strongly continuous in $L^p(X,\sigma)$.
To this aim, notice that for $f \in C_b(X)$ the function $(t, x) \mapsto P_t f(x)$ is continuous in
$[0, \infty)\times X$ (see \cite[Lemma 2.1]{Bog3}). This fact, estimate $\|P_t f\|_\infty\le \|f\|_\infty$
together with the dominated convergence theorem imply that
$\|P_t f-f\|_{L^p(X,\sigma)}$ vanishes as $t \to 0^+$ for any $f \in C_b(X)$ and $p \in [1,\infty)$.
To conclude we argue by approximation. Let $ f \in L^p(X,\sigma)$ and $(f_n)_n \subseteq C_b(X)$ converging
to $f$ in $L^p(X,\sigma)$ as $n\to \infty$. Then,
\begin{align}\label{c_0}
\|P_t f-f\|_{L^p(X,\sigma)}&\le \|P_t (f-f_n)\|_{L^p(X,\sigma)}+\|P_t f_n-f_n\|_{L^p(X,\sigma)}
+\|f_n-f\|_{L^p(X,\sigma)}
\notag\\
& \le \|P_t f_n-f_n\|_{L^p(X,\sigma)}+2\|f_n-f\|_{L^p(X,\sigma)},
\end{align}
where in the last line we used the contractivity of $P_t$ in $L^p(X,\sigma)$. Fix $\varepsilon>0$ and
let $n_0\in \N$ be such that $\|f_{n_0}-f\|_{L^p(X,\sigma)}\le \varepsilon/4$. The first part of the
proof yields the existence of $t_0>0$ such that $\|P_t f_{n_0}-f_{n_0}\|_{L^p(X,\sigma)}\le \varepsilon/2$
for any $t \in (0,t_0)$. Thus, writing estimate \eqref{c_0} with $n$ replaced by $n_0$, we conclude that
$\|P_t f-f\|_{L^p(X,\sigma)}\le \varepsilon$ for any $t \in (0, t_0)$ and this completes the proof.
\end{proof}

The next result concerns the asymptotic behaviour of $P_t$ as $t\to\infty$. For any
$f \in L^1(X,\sigma)$ we denote by $m_{\sigma}(f)$
the mean of $f$ with respect to $\sigma$, i.e.,
\begin{equation*}
m_\sigma(f):= \int_X fd\sigma.
\end{equation*}

\begin{lemm}
Assume that Hypotheses $\ref{hyp_App}$ and $\ref{base}$ hold true. For any $f\in {\rm Lip}_b(X)$, $P_t f$ converges pointwise to $m_\sigma(f)$ as $t \to \infty$, i.e.,
\begin{equation*}
\lim_{t \to \infty}P_t f(x)= m_\sigma(f),\qquad x \in X;
\end{equation*}
and if $f>0$ then
\begin{equation}\label{asympt}
\lim_{t \to \infty}\int_X (P_t f)\log(P_t f) d\sigma= m_\sigma(f)\log(m_\sigma(f)).
\end{equation}
\end{lemm}
\begin{proof}
Let $f:X \to \mathbb{R}$ be a bounded Lipschitz continuous function, $x \in X$ and $\varepsilon>0$.
The weak-star convergence of $\mu_t$ to $\sigma$ as $t\to\infty$ (see Theorem \ref{furhman})
implies that there exists $t_0>0$ such that for every $t\geq t_0$
\begin{align}\label{erg1}
\abs{\int_Xfd\mu_t-\int_Xfd\sigma}<\frac{\eps}{2}.
\end{align}
Analogously, as $T_tx$ vanishes as $t\to\infty$, there is $t_1>0$ such that for every
$t\geq t_1$
\begin{align}\label{erg2}
|T_tx|<\frac{\eps}{2L},
\end{align}
where $L$ is the Lipschitz constant of $f$.
Thus, for every $t\geq \max\{t_0, t_1\}$, by \eqref{erg1} and \eqref{erg2}, we have
\begin{align*}
\left|P_tf(x)-\int_Xfd\sigma\right|&=\abs{\int_Xf(T_tx+z)\mu_t(dz)-\int_Xf(z)\sigma(dz)}
\\
&\leq \abs{\int_X\left(f(T_tx+z)-f(z)\right)\mu_t(dz)}+\abs{\int_Xf(z)\mu_t(dz)-\int_Xf(z)\sigma(dz)}
\\
&\leq \int_X\left|f(T_tx+z)-f(z)\right|\mu_t(dz)+\frac{\eps}{2}
\\
&\le L \int_X|T_tx|\mu_t(dz)+\frac{\eps}{2} <\eps.
\end{align*}
If $f>0$ the function $x \mapsto (P_t f)(x)\log((P_t f)(x))$ is well defined and bounded by
$\|f\|_\infty\log\|f\|_\infty$, as $P_t$ is contractive and preserves positivity. Therefore, by the dominated
convergence theorem we get \eqref{asympt}.
\end{proof}

\begin{cor}
Assume that Hypotheses $\ref{hyp_App}$ and $\ref{base}$ hold true. For any $f \in L^p(X,\sigma)$, it holds that
\begin{equation}\label{conv_lp}
\lim_{t \to \infty}\|P_t f-m_\sigma(f)\|_{L^p(X,\sigma)}=0.
\end{equation}
\end{cor}
\begin{proof}
Formula \eqref{conv_lp} easily follows from the dominated convergence theorem for bounded and Lipschitz
continuous functions. The general case follows by approximation from Proposition \ref{approx}.
\end{proof}

\section{A logarithmic Sobolev type inequality and its consequences}\label{sec_log}

In this section we prove a logarithmic Sobolev type inequality satisfied by
$\sigma$, the unique invariant measure for $P_t$ provided by Theorem \ref{furhman}. For any positive function
on $X$ we denote by
$$
{\rm Ent}_\sigma(f):=\pa{\int_X f \log f d\sigma}- m_\sigma(f)\log m_\sigma (f)
$$
the entropy of $f$ with respect to $\sigma$.
In order to prove the desired logarithmic Sobolev type inequality we need further assumptions on the L\'{e}vy
measure $M$. Similar assumptions are considered in \cite[Hypothesis (H4) and Lemma 2.1(2)]{RW03}.

\begin{hyp}\label{hyp_wang}
We assume that $M$ is a L\'{e}vy measure on $X$ and
\begin{enumerate}[\rm (i)]
\item\label{hyp_wang1} there exists a function $h:(0,\infty)\ra(0,\infty)$ such that, for every $t>0$,
\[
h(t)M-M\circ T_t^{-1}
\]
is a positive measure;
\item the function $h$ belongs to $L^1((0,\infty))$.
\end{enumerate}
\end{hyp}

In the following lemma we prove an estimate that plays the role of pointwise gradient estimates
in the local case.

\begin{lemm}
Assume that Hypotheses \ref{hyp_App}, \ref{base} and \ref{hyp_wang}\eqref{hyp_wang1} hold true. Then,
\begin{align}\label{ineq_wang}
\int_X\Big|(P_tf)(x+y)-(P_tf)(x)\Big|^p M(dy)\leq
h (t)\int_X\Big|P_t\Big(f(\cdot+y)-f(\cdot)\Big)(x)\Big|^p M(dy).
\end{align}
for every $p \in [1,\infty)$, $t>0$, $x\in X$ and $f\in B_b(X)$.
\end{lemm}
\begin{proof}
Using \eqref{mehler}, the Jensen inequality and Hypothesis \ref{hyp_wang}\eqref{hyp_wang1} we get
\begin{align*}
&\int_X\Big|(P_tf)(x+y)-(P_tf)(x)\Big|^p M(dy)
\\
=&\int_X\abs{\int_X\Big (f(T_t x+ T_ty+z)-f(T_t x+z)\Big)d\mu_t(z)}^p M(dy)
\\
=&\int_X\abs{\int_X\Big (f(T_t x+w+z)-f(T_t x+z)\Big)\mu_t(dz)}^p d(M\circ T_t^{-1})(w)
\\
\leq & h (t)\int_X\abs{\int_X\Big (f(T_t x+w+z)-f(T_t x+z)\Big)\mu_t(dz)}^p M(dw)
\\
=&h (t)\int_X|P_t\big(f(\cdot+w)-f(\cdot)\big)(x)|^p M(dw).\qedhere
\end{align*}
\end{proof}

The main result of this section is the following estimate of the entropy of $f$. As pointed out in the
Introduction, on the right hand side the gradient of $f$, typical of the logarithmic Sobolev inequalities
available in the context of semigroups generated by {\em local} operators, has to be replaced by the integral of the relative increment because of the {\em nonlocal} effects.

\begin{thm}\label{thm_neve}
Assume that Hypotheses $\ref{hyp_App}$, $\ref{base}$ and $\ref{hyp_wang}$ hold true. Then, for every
$p \in [1,\infty)$ and $f \in \mathcal{F}C^2_A(X)$ with positive infimum, the following estimate
\begin{align}\label{log_sob}
{\rm Ent}_\sigma(f^p)\leq
C\int_X\int_X\frac{|f^p(x+y)-f^p(x)|^2}{f^p(x)} M(dy)\sigma(dx),
\end{align}
holds true with $C= \|h \|_{L^1((0,\infty))}$.
\end{thm}
\begin{proof}
Let $f$ be as in the statement. It is not restrictive to assume also that $\sup f \le 1$. Indeed,
if this is not the case we consider $f/\|f\|_\infty$ in place of $f$.
Thus, consider the function $F:(0,\infty)\ra\R$ defined as
\[
F(t):=\int_X (P_t f^p)\log(P_t f^p)d\sigma.
\]
The function $(t,x)\mapsto \psi(t,x):=(P_t f^p(x))\log(P_t f^p(x))$ is bounded and continuously
differentiable in $[0,\infty)\times X$ since $P_t f^p$ belongs to $\mathcal{F}C^2_A(X)$ and
takes values in $[(\inf f)^p, 1]$ for any $t,x$ as above (see \eqref{mehler}). Moreover, since
$$
\frac{\partial}{\partial t}\psi(t,\cdot)=(\log(P_t f^p)+1)\mathcal{L}(P_t f^p)
$$
belongs to $\mathcal{F}C_b(X)$, the function $F$ is differentiable and its derivative is given by
\begin{align*}
F'(t)&=\int_X (\mathcal{L}(P_t f^p))\log(P_t f^p)d\sigma+\int_X \mathcal{L}(P_t f^p) d\sigma=\int_X (\mathcal{L}(P_t f^p))\log(P_t f^p)d\sigma
\\
&=-\int_X\int_X \bigg[(P_tf^p)(x+y)\log (P_tf^p)(x+y)-(P_tf^p)(x+y)-(P_tf^p)(x)\log (P_tf^p)(x)
\\
&\qquad\qquad\phantom{a1}+(P_tf^p)(x)-\Big((P_tf^p)(x+y)-(P_tf^p)(x)\Big)\log (P_tf^p)(x)\bigg]M(dy)\sigma(dx).
\end{align*}
where we used \eqref{inv_L} and \eqref{form_eu} with $\Phi(\xi)= \xi\log \xi-\xi$. Notice that
for every $r,s>0$ the following inequality holds
\[
r\log r-r-s\log s+s-(r-s)\log r\leq \frac{(r-s)^2}{s}.
\]
Indeed, multiplying by $s^{-1}$ and setting $t:=rs^{-1}$ it is reduced to the elementary estimate
$\log t\leq t-1$, $t >0$. By this last inequality and \eqref{ineq_wang} we get
\begin{align}\label{Orpheus}
F'(t)&\geq -\int_X\frac{1}{(P_tf^p)(x)}\int_X\Big((P_tf^p)(x+y)-(P_tf^p)(x)\Big)^2 M(dy)\sigma(dx)
\notag\\
&\geq -h (t)\int_X\frac{1}{(P_tf^p)(x)}\int_X\left(P_t\Big(f^p(\cdot+y)-f^p\Big)(x)\right)^2M(dy)\sigma(dx).
\end{align}
The H\"older inequality yields
\begin{align}\label{holder_Pt}
\left|P_t\Big(f^p(\cdot+y)-f^p\Big)(x)\right|\leq
\left(P_t\left(\frac{|f^p(\cdot+y)-f^p|^2}{f^p}\right)(x)\right)^{1/2}\pa{P_tf^p(x)}^{1/2},
\end{align}
for every $x,y \in X$, hence combining \eqref{Orpheus} with \eqref{holder_Pt} we get
\begin{align*}
F'(t)&\geq -h (t)\int_X\int_XP_t\left(\frac{|f^p(\cdot+y)-f^p|^2}{f^p}\right)(x) M(dy)\sigma(dx).
\end{align*}
By the Fubini theorem and the invariance of $\sigma$ with respect to $P_t$ we get
\begin{align*}
F'(t)&\geq -h (t)\int_X\int_X\frac{|f^p(x+y)-f^p(x)|^2}{f^p(x)} M(dy)\sigma(dx).
\end{align*}
Now integrating the previous inequality from $0$ to $t$ we get
\begin{align*}
F(t)-F(0)
& \ge -\|h \|_{L^1((0,\infty))}\int_X\int_X\frac{|f^p(x+y)-f^p(x)|^2}{f^p(x)} M(dy)\sigma(dx)
\end{align*}
or equivalently
\[
\int_X \! P_t f^p\log(P_t f^p)d\sigma - \! \int_X f^p\log f^p d\sigma
\geq -\|h \|_{L^1((0,\infty))} \! \int_X \! \int_X \! \frac{|f^p(x+y)-f^p(x)|^2}{f^p(x)} M(dy)\sigma(dx).
\]
Letting $t$ to infinity, and recalling \eqref{asympt} we get
\begin{align*}
\int_X f^p\log f^pd\sigma-\pa{\int_X f^pd\sigma}\log\pa{\int_X f^pd\sigma}
\leq C\int_X\int_X\frac{|f^p(x+y)-f^p(x)|^2}{f^p(x)} M(dy)\sigma(dx),
\end{align*}
where $C:= \|h \|_{L^1((0,\infty))}$, whence the claim.
\end{proof}

Now, let us denote by $\mathcal{H}^p$ the Banach space completion of $\mathcal{F}C^2_A(X)$
with respect to the norm
\begin{align*}
\norm{f}_{\mathcal{H}^p}:=\|f\|_{L^p(X,\sigma)}+\pa{\int_X\int_X|f(x+y)-f(x)|^2M(dy)\sigma(dx)}^{1/2}.
\end{align*}
Observe that, since $M$ is a L\'{e}vy measure, for every $f$ belonging to $\mathcal{F}C^2_A(X)$
$$
\|f\|_{\mathcal{H}^p}\le (1+2\sqrt{M(B^1_c)})\|f\|_\infty+
\|D f\|_\infty\left(\int_{B_1}|y|^2 M(dy)\right)^{1/2}<\infty.
$$
An immediate consequence of \eqref{log_sob} is the Poincar\'{e} inequality \eqref{est-poi}.
Similar estimates have already been proved in \cite[Corollary 1.4]{RW03}. But, we
derive them from the logarithmic Sobolev type inequality \eqref{log_sob}, while in \cite{RW03},
as these were not available when $M \not\equiv 0$, they are derived by using an idea due to Bakry
and Ledoux which consists in differentiating the map $s\mapsto P_{t-s}(P_s f)^2$
(see \cite{BL}) in order to get \eqref{est-poi}.

\begin{prop}\label{Poincare}
Under the hypotheses of Theorem \ref{thm_neve}, the estimate
\begin{equation}\label{est-poi}
\|f-m_\sigma(f)\|_{L^2(X, \sigma)}\le \sqrt{2C}\left(\int_X \int_X |f(x+y)-f(x)|^2M(dy)\sigma(dx)\right)^{1/2}
\end{equation}
holds true for any $f \in \mathcal{H}^2$. Here $C$ is the constant appearing in \eqref{log_sob}.
\end{prop}
\begin{proof}
Consider first $f \in \mathcal{F}C^2_A(X)$ with $m_\sigma(f)=0$. For
$0<\varepsilon <(2\|f\|_\infty)^{-1}$, the function $f_\varepsilon:=1+\varepsilon f$ is greater or equal to
$1/2$. Thus, estimate \eqref{log_sob} with $p=2$ yields
\begin{equation*}\int_X f_\varepsilon^2\log(f_\varepsilon^2)d\sigma-m_\sigma(f_\varepsilon^2)
\log(m_\sigma(f_\varepsilon^2))\leq
C\int_X\int_X\frac{|f_\varepsilon^2(x+y)-f_\varepsilon^2(x)|^2}{f_\varepsilon^2(x)} M(dy)\sigma(dx).
\end{equation*}
Observing that
$$
\int_X f_\varepsilon^2\log(f_\varepsilon^2)d\sigma-m_\sigma(f_\varepsilon^2)\log(m_\sigma(f_\varepsilon^2))
= 2\varepsilon^2\|f\|_{L^2(X,\sigma)}^2+o(\varepsilon^2),\qquad\;\, \varepsilon \to 0^+ ,
$$
and that
\begin{align*}
\int_X&\int_X\frac{|f_\varepsilon^2(x+y)-f_\varepsilon^2(x)|^2}{f_\varepsilon^2(x)}M(dy)\sigma(dx)
\\
&=\int_X\int_X \frac{\Big(\varepsilon^2(f^2(x+y)-f^2(x))+2\varepsilon
(f(x+y)-f(x))\Big)^2}{(1+\varepsilon f(x))^2}M(dy)\sigma(dx)
\\
&= 4\varepsilon^2\int_X\int_X\frac{|f(x+y)-f(x)|^2}{(1+\varepsilon f(x))^2}M(dy)\sigma(dx)+
\int_X\int_X \frac{g_\varepsilon(x,y)}{(1+\varepsilon f(x))^2}M(dy)\sigma(dx),\qquad\;\,
\end{align*}
where
$$
g_\varepsilon(x,y)=\varepsilon^4(f^2(x+y)-f^2(x))^2+4\varepsilon^3(f^2(x+y)-f^2(x))(f(x+y)-f(x)),
$$
we get
\begin{align}\label{esami}
2\varepsilon^2\|f\|_{L^2(X,\sigma)}^2+o(\varepsilon^2)\le&
4\varepsilon^2C\int_X\int_X\frac{|f(x+y)-f(x)|^2}{(1+\varepsilon f(x))^2}M(dy)\sigma(dx)
\notag\\
&+C\int_X\int_X \frac{g_\varepsilon(x,y)}{(1+\varepsilon f(x))^2}M(dy)\sigma(dx).
\end{align}
Using the assumptions on $f$ and $\varepsilon$ we can estimate
$$
\frac{|f(x+y)-f(x)|^2}{f_\varepsilon(x)^2}\le \frac{1}{4}\Bigl(2\|f\|_\infty^2 \chi_{B^c_1}(y)+
\|Df\|_\infty^2|y|^2\chi_{B_1}(y)\Bigr),\qquad\;\, x,y \in X .
$$
and, analogously
$$
\frac{g_\varepsilon(x,y)}{f_\varepsilon(x)^2}\le \frac{1}{4}
\Bigl(C_1\chi_{B^c_1}(y)+C_2|y|^2\chi_{B_1}(y)\Bigr),
\qquad\;\, x,y \in X ,
$$
for some positive constants $C_1$ depending on $\|f\|_\infty$ and $C_2$ depending on $\|f\|_\infty$
and $\|Df\|_\infty$.
As $M$ is a L\'evy measure, letting $\varepsilon \to 0$ in \eqref{esami} by the dominated convergence
theorem we get
\begin{equation}\label{poi_est}
\|f\|_{L^2(X,\sigma)}^2\le 2C \int_X\int_X |f(x+y)-f(x)|^2 M(dy)\sigma(dx).
\end{equation}
For a general $f \in \mathcal{F}C^2_A(X)$, applying \eqref{poi_est} to $f -m_\sigma(f)$, we deduce
\begin{equation}\label{poi_est-1}
\|f-m_\sigma(f)\|_{L^2(X,\sigma)}^2\le 2C \int_X\int_X |f(x+y)-f(x)|^2 M(dy)\sigma(dx).
\end{equation}
To conclude, let us consider $f \in \mathcal{H}^2$ and let $(f_n)\subseteq \mathcal{F}C^2_A(X)$ converging
to $f$ in $\norm{\cdot}_{\mathcal{H}^2}$. Then writing \eqref{poi_est-1} with $f_n$ in place of $f$ and
letting $n \to \infty$ we get the claim. Indeed
$$
\|f_n-m_\sigma(f_n)\|_{L^2(X,\sigma)}^2= \|f_n\|_{L^2(X,\sigma)}^2-(m_\sigma(f_n))^2
$$
converges to $\|f\|_{L^2(X,\sigma)}^2-(m_\sigma(f))^2$ by the dominated convergence theorem.
Further, since
$$
\Big||f_n(x+y)-f_n(x)|-|f(x+y)-f(x)|\Big|\le |f_n(x+y)-f(x+y)-f_n(x)+f(x)|,
$$
for $M$-a.e. $y\in X$, $\sigma$-a.e. $x \in X$ and the right hand side of the previous inequality vanishes as
$n \to \infty$ for $M$-a.e. $y\in X$ and $\sigma$-a.e. $x \in X$, coming back to \eqref{poi_est-1} with
$f_n$ in place of $f$ and letting $n\to\infty$ we conclude the proof.
\end{proof}

For $p\in[1,\infty)$, we denote by
\[
\mathcal{W}^p=\left\{f:X\to \R\,\middle|\,\int_X\int_X ||f|^p(x+y)-|f|^p(x)|M(dy)\sigma(dx)<\infty\right\}.
\]
In the following proposition we use a bootstrap procedure similar the one in \cite{ALL13} in order to obtain estimates looking like \eqref{est-poi} for $p>2$.

\begin{prop}\label{prop-est-2}
Assume Hypotheses \ref{hyp_App} and \ref{base} hold true. For any $f \in \mathcal{W}^2$ it holds that
\begin{equation}\label{est-2}
\|f-m_\sigma(f)\|_{L^2(X, \sigma)}\le c\left(\int_X \int_X ||f|^2(\cdot+y)-|f|^2|M(dy)d\sigma\right)^{1/2}
\end{equation}
for some positive constant $c$. Then, for every $p \in (2,\infty)$, there exists a positive constant $c_p$
such that
\begin{equation}\label{est-p}
\|f\|_{L^p(X, \sigma)}^p\le c_p\int_X \int_X \Big||f|^p(x+y)-|f|^p\Big|\big[\chi_{B_1^c}(y)
+|y|^{2-p}\chi_{B_1}(y)\big]M(dy)\sigma(dx)
\end{equation}
for any $f \in \mathcal{W}^p$ with $m_\sigma(f)=0$.
\end{prop}
\begin{proof}
Let $f\in \mathcal{W}^p$ with $m_\sigma(f)=0$. Since $p>2$, the function $f^{p/2}$ belongs to
$\mathcal{W}^2$. Then, applying estimate \eqref{est-2} to $f^{p/2}$ we deduce
\begin{align}\label{mia}
\|f\|_{L^p(X,\sigma)}^p -\|f\|_{L^{p/2}(X,\sigma)}^p&=
\|f^{p/2}-m_\sigma(f^{p/2})\|_{L^2(X, \sigma)}^2\notag
\\
&\le c \int_X \int_X ||f|^p(x+y)-|f|^p(x)|M(dy)\sigma(dx).
\end{align}
Now, if $p\in (2, 4]$, using that $ \|f\|_{L^{p/2}(X,\sigma)}\le  \|f\|_{L^2(X,\sigma)}$,
from \eqref{est-2} and \eqref{mia} we obtain
\begin{align}\label{int}
\|f\|_{L^p(X,\sigma)}^p &\le
\|f\|_{L^{2}(X,\sigma)}^p+c \int_X \int_X ||f|^p(x+y)-|f|^p(x)|M(dy)\sigma(dx)
\\
&\le c\left(\int_X \int_X ||f|^2(x+y)-|f|^2(x)|M(dy)\sigma(dx)\right)^{p/2}
\notag\\
&\quad+c \int_X \int_X ||f|^p(x+y)-|f|^p(x)|M(dy)\sigma(dx)
\notag\\
& \le  c\int_X \left(\int_X ||f|^2(x+y)-|f|^2(x)|M(dy)\right)^{p/2} \sigma(dx)
\notag\\ \notag
&\quad+c \int_X \int_X ||f|^p(x+y)-|f|^p(x)|M(dy)\sigma(dx)
\end{align}
where in the last line we used the Jensen inequality taking into account that $\sigma$ is a
Borel probability measure.
Furthermore, multiplying and dividing by $|y|^{2(p-2)/p}$ in $B_1$, using the H\"{o}lder inequality
and using that $M$ is a L\'{e}vy measure,  we can estimate
\begin{align*}
&\left(\int_X  ||f|^2(x+y)-|f|^2(x)|M(dy)\right)^{p/2}
\\
&\le 2^{\frac{p-2}{2}} \left(\int_{B_1} ||f|^2(x+y)-|f|^2(x)|M(dy)\right)^{p/2}
\notag\\
&\quad+ 2^{\frac{p-2}{2}} \left(\int_{B_1^c} ||f|^2(x+y)-|f|^2(x)|M(dy)\right)^{p/2}
\notag\\
& \le  2^{\frac{p-2}{2}} \left(\int_{B_1} |y|^2M(dy)\right)^{\frac{p-2}{2}}
\int_{B_1} \frac{||f|^2(x+y)-|f|^2(x)|^{p/2}}{|y|^{p-2}}M(dy)
\notag\\
&\quad+ (2M(B_1^c))^{\frac{p-2}{2}} \int_{B_1^c} ||f|^2(x+y)-|f|^2(x)|^{p/2}M(dy)
\notag\\
& \le  2^{\frac{p-2}{2}} \left(\int_{B_1} |y|^2M(dy)\right)^{\frac{p-2}{2}}
\int_{B_1} \frac{||f|^p(x+y)-|f|^p(x)|}{|y|^{p-2}}M(dy)
\notag\\
&\quad+(2M(B_1^c))^{\frac{p-2}{2}} \int_{B_1^c} ||f|^p(x+y)-|f|^p(x)|M(dy)
\end{align*}
where in the last inequality we used estimate $||a|-|b||^p\le ||a|^p-|b|^p|$ which holds true for
$a,b \in \R$, $p>1$. Indeed, assuming $|a|>|b|>0$ and setting $t=|a||b|^{-1}$, it suffices
to prove that $(1,\infty)\ni t\mapsto g(t):=(t-1)^p-t^p+1$ is nonpositive. But,
$g'(t)=p[(t-1)^{p-1}-t^{p-1}]\geq 0$ and then $g(t)\le g(1)=0$ for any $t \in (1,\infty)$.
Thus, summing up in \eqref{int} we get \eqref{est-p} for $p \in (2, 4]$.
Now, let $p \in (4, 8]$, then $p/2 \in (2, 4].$ Thus, starting from \eqref{mia} and using estimate \eqref{est-p} with $p/2$ in place of $p$ we deduce
\begin{align*}
\|f\|_{L^p(X,\sigma)}^p \le&
\|f\|_{L^{p/2}(X,\sigma)}^p + C \int_X\int_X ||f|^p(x+y)-|f|^p(x)|M(dy)\sigma(dx)
\\
\le& c_{p/2}^2\Bigg(\int_X\int_{B_1^c} ||f|^{p/2}(x+y)-|f|^{p/2}(x)|M(dy)\sigma(dx)
\\
&+\int_X\int_{B_1} \frac{||f|^{p/2}(x+y)-|f|^{p/2}(x)|}{|y|^{\frac{p}{2}-2}}M(dy)\sigma(dx)\Bigg)^2
\\
&+ C \int_X\int_X ||f|^p(x+y)-|f|^p(x)|M(dy)\sigma(dx)
\\
\le& 2 c_{p/2}^2\Bigg( \int_X\int_{B_1^c} ||f|^{p/2}(x+y)-|f|^{p/2}(x)|M(dy)\sigma(dx)\Bigg)^2
\\
&+ 2c_{p/2}^2\Bigg(\int_X\int_{B_1}
\frac{||f|^{p/2}(x+y)-|f|^{p/2}(x)|}{|y|^{\frac{p}{2}-2}}M(dy)\sigma(dx)\Bigg)^2
\\
&+ C \int_X\int_X ||f|^p(x+y)-|f|^p(x)|M(dy)\sigma(dx)
\\
\le& 2 c_{p/2}^2 \int_X\int_{B_1^c} ||f|^{p}(x+y)-|f|^{p}(x)|M(dy)\sigma(dx)
\\
&+  2c_{p/2}^2\Big(\int_{B_1}|y|^2M(dy)\Big)^{1/2}\int_X\int_{B_1}
\frac{||f|^{p}(x+y)-|f|^{p}(x)|}{|y|^{p-2}}M(dy)\sigma(dx)
\\
& + C \int_X\int_X ||f|^p(x+y)-|f|^p(x)|M(dy)\sigma(dx)
\end{align*}
getting again \eqref{est-p} for $p\in (4,8]$. Iterating this procedure we complete the proof.
\end{proof}

\begin{rmk}{\rm
Note that \eqref{est-2} is implied by \eqref{est-poi}. Therefore, under the hypotheses of Theorem
\ref{Poincare}, inequality \eqref{est-p} holds true.
}\end{rmk}

The arguments used in the proof of Proposition \ref{prop-est-2} can be used to deduce the
integrability of functions with polynomial growth with respect to $\sigma$ from their
integrability with respect to $M$. We discuss this in term of the moments of $M$ and $\sigma$.
If $\mu$ is a Borel measure on $X$, we denote by
\[
\mu_A(p):=\int_A |x|^p \mu(dx),\qquad\;\, A\in \mathcal{B}(X)
\]
the moment of order $p$ of $\mu$ on $A$. In the case $A=X$ we simply write $\mu(p)$.

\begin{prop}
Assume that the hypotheses of Theorem $\ref{thm_neve}$ are satisfied and that $\sigma(1),M(1)<\infty$. Then (i) if $M(2)<\infty$, then $\sigma(2)<\infty$ and (ii) if $\sigma(2)<\infty$ and $M(p)<\infty$ for some $p>2$, then $\sigma(p) <\infty$.
\end{prop}
\begin{proof}
First of all, observe that, since $M$ is a L\'evy measure, the assumption $M(p)<\infty$
is equivalent to $M_{B^c_1}(p)<\infty$. Let first be $p=2$. Setting $f(x)=|x|$, observe that the function
$f^p$ is convex and differentiable for any $p\ge2$ and for every $x,y \in X$ satisfies
$$
|f^p(x+y)-f^p(x)|\le \max\{|\langle D f^p(x+y),y\rangle|, |\langle D f^p(x),y\rangle|\}
\le |y|\bigl(|D f^p(x+y)|+ |D f^p(x)|\bigr).
$$
We deduce, for $p\geq 2$,
\begin{align}\label{ricor}
\int_X\int_X ||x+y|^p-|x|^p|M(dy)\sigma(dx)
&\le p \int_X\int_X|y|(|x+y|^{p-1}+|x|^{p-1})M(dy)\sigma(dx)
\notag\\
&\le p(1+2^{p-2})M(1)\sigma(p-1)+2^{p-2}pM(p).
\end{align}
Estimates \eqref{mia} and \eqref{ricor} prove assertion (i).

Now, take $p>2$. Applying \eqref{int} to $f(x)=|x|$ and using \eqref{ricor} we have
\begin{align}\label{p24}
\sigma(p)\le &\sigma(2)^{p/2}+cp(1+2^{p-2})M(1)\sigma(p-1)+cp2^{p-2}M(p).
\end{align}
The Young inequality yields
\[
\sigma(p-1)\le \varepsilon \frac{p-1}{p}\sigma(p)+\frac{1}{p\varepsilon^{p-1}}
\]
for any $\varepsilon>0$. Hence, from \eqref{p24} we get
\[
\sigma(p)\le \sigma(2)^{p/2}+c(1+2^{p-2})\varepsilon(p-1) M(1) \sigma(p)+
\frac{c(1+2^{p-2})}{\varepsilon^{p-1}}M(1)+2^{p-2}cpM(p),
\]
and choosing $\varepsilon>0$ small enough we conclude that
\[
\sigma(p)\le  C_1(\sigma(2))^{p/2}+ C_2M(p) + C_3M(1),
\]
for some positive constants $C_1$, $C_2$ and $C_3$.
\end{proof}

Now we introduce an appropriate class of functions that satisfies \eqref{log_sob} in a ``nicer'' way.
For every $c>0$ consider the space
\begin{align*}
[\mathcal{F}C^2_A(X)]_c:=\{f\in\mathcal{F}C^2_A(X)\,|\,|f|\geq c\}.
\end{align*}
Let $\mathcal{H}^2_c$ be the closure of $[\mathcal{F}C^2_A(X)]_c$ in $\mathcal{H}^2$ and let us set
\[
\mathcal{H}^2_0:=\bigcup_{c>0}\mathcal{H}^2_c.
\]

\begin{prop}\label{ciao}
Under the hypotheses of Theorem \ref{thm_neve}, for any $f\in\mathcal{H}^2_0$ there exists a positive constant $c_f$, depending on $f$, such that
\begin{align}\label{log_sob_est}
{\rm Ent}_\sigma(|f|)\leq Cc_f\int_X\int_X|f(x+y)-f(x)|^2 M(dy)\sigma(dx).
\end{align}
where $C$ is defined in Theorem \ref{thm_neve}.
\end{prop}
\begin{proof}
First note that from \eqref{log_sob} we immediately deduce estimate \eqref{log_sob_est} for every
$f$ in $\mathcal{F}C^2_A(X)$ with positive infimum. In this case estimate \eqref{log_sob_est}
holds true with $c_f^{-1}=\inf f$. To deduce \eqref{log_sob_est} for a general $f$ in
$[\mathcal{F}C^2_A(X)]_c$, consider the sequence $(f_n)$ defined by $f_n:=(f^2+(1/n))^{1/2}$ satisfying
$f_n\ge c$ for every $n\in\N$. Writing \eqref{log_sob_est} with $f_n$ in place of $f$ we get
$$
\int_X f_n\log (f_n)d\sigma-m_\sigma(f_n)\log (m_\sigma(f_n))\leq
Cc^{-1}\int_X\int_X|f_n(x+y)-f_n(x)|^2 M(dy)\sigma(dx).
$$
Observing that $f_n$ converges pointwise to $|f|$ as $n \to \infty$, by the dominated convergence theorem we
deduce estimate \eqref{log_sob_est} for functions in $[\mathcal{F}C^2_A(X)]_c$.

Now let $f\in \mathcal{H}^2_0$. Then, there exists $c>0$ such that $f\in \mathcal{H}^2_c$ and a sequence
$(f_n)_{n\in\N}\subseteq [\mathcal{F}C^2_A(X)]_c$ converging to $f$ in $\mathcal{H}^2$, as $n\to\infty$. By the previous step
\begin{align*}
\int_X |f_n|\log |f_n|d\sigma-m_\sigma(|f_n|)\log(m_\sigma(|f_n|))\leq
Cc^{-1}\int_X\int_X|f_n(x+y)-f_n(x)|^2 M(dy)\sigma(dx).
\end{align*}
Up to a subsequence we may assume that $(f_n)_{n\in\N}$ converges pointwise $\sigma$-a.e. to $f$.
So by the Fatou lemma we have
\begin{align*}
\int_X|f|\log|f|d\sigma\leq \liminf_{n\ra\infty}\int_X |f_n|\log |f_n|d\sigma
\end{align*}
and using the convergence in $\mathcal{H}^2$ we obtain
\begin{align*}
\int_X|f|\log|f|d\sigma &\leq \liminf_{n\ra\infty}\int_X |f_n|\log |f_n|d\sigma
\\
&\leq\liminf_{n\ra\infty}\sq{\pa{\int_X |f_n|d\sigma}\log\pa{\int_X |f_n|d\sigma}}
\\
&\quad+Cc^{-1}\liminf_{n\ra\infty}\sq{\int_X\int_X|f_n(x+y)-f_n(x)|^2 M(dy)\sigma(dx)}
\\
&= m_\sigma(|f|)\log(m_\sigma(|f|))+Cc^{-1}\int_X\int_X|f(x+y)-f(x)|^2 M(dy)\sigma(dx)
\end{align*}
and we conclude.
\end{proof}

\begin{rmk}\label{gaussiana}{\rm
Until now, we have assumed $Q=0$ in the representation $\lambda\leftrightarrow [b,Q,M]$ of the
function defined in \eqref{lambda}. If $Q\neq 0$ is a nonnegative self-adjoint trace-class operator, then
we define the operators $Q_t = \int_0^t T_sQT_s^*ds$, that are nonnegative, self-adoint and trace-class
as well. The measures $\mu_t$ are associated with the triple $[b_t,Q_t,M_t]$ with $b_t,M_t$ given by
\eqref{scrivere meglio}, \eqref{defM_t}. Assuming that $\sup_t {\rm Tr}\,Q_t <\infty$, the operator
$Q_\infty$ is well defined and under the assumptions of Theorem \ref{furhman} there is a unique
invariant measure associated with $P_t$ given by the convolution between the Gaussian measure
$\gamma:=\mathcal{N}(0,Q_\infty)\leftrightarrow[0,Q_\infty,0]$ and the probability
measure $\sigma\leftrightarrow[b_\infty,0,M_\infty]$. In such case, assume further
the estimate
\begin{equation}\label{gra-est}
|Q_\infty^{1/2} DP_tf|\le \psi(t) P_t (| Q_\infty^{1/2}Df|),
\end{equation}
for any $t \ge 0$, $f \in \mathcal{F}C^2_A(X)$ and some nonnegative $\psi \in L^1((0,\infty))$
(see, for instance, the proof of \cite[Proposition 11.2.17]{DaP02} for the Ornstein-Uhlenbeck semigroup and \cite[Lemma 2.1]{RW03} for more general semigroups in the infinite dimensional case and \cite[Chapter 6]{Lor} for the same estimates in finite dimension).
Then, by the classical logarithmic Sobolev and Poincar\'e inequalities for $\gamma$
and the product property of the entropy and the variance with respect to
convolution of measures (see \cite[Proposition 2.2]{Led}), estimates \eqref{log_sob} and
\eqref{poi_est} can be reformulated as
\begin{equation}\label{bing}
{\rm Ent}_{\gamma * \sigma}(f^p)\le c\int_{X} f^{p-2}| Q_\infty^{1/2} Df|^2 d\gamma
+ C\int_X\int_X \frac{|f^p(x+y)-f^p(x)|^2}{f^p(x)} M(dy)\sigma(dx)
\end{equation}
for any $p \in [1,\infty)$, $f\in \mathcal{F}C^2_A(X)$ with positive infimum and some positive $c$
depending on $\|\psi\|_{L^1((0,\infty))}$, and
\begin{equation*}
\|f-m_\sigma(f)\|_{L^2(X,\gamma*\sigma)}\le c'\|Q^{1/2}Df\|_{L^2(X,\gamma)}+
\sqrt{2C}\left(\int_X \int_X |f(x+y)-f(x)|^2M(dy)\sigma(dx)\right)^{1/2}
\end{equation*}
for any $f\in \mathcal{H}^2$ and some positive $c'$.
In particular, estimate \eqref{log_sob_est} becomes
\begin{equation*}
{\rm Ent}_{\gamma * \sigma}(|f|)\le c\int_{X} |f|^{-1}| Q_\infty^{1/2} D|f||^2 d\gamma+
Cc_f\int_X\int_X |f(x+y)-f(x)|^2 M(dy)\sigma(dx)
\end{equation*}
for any $f\in \mathcal{H}^2_0$ and some positive constant $c'$. Here $C$ and $c_f$ are the constants appearing in Theorem \ref{thm_neve}.}
\end{rmk}

\section{Exponential integrability of Lipschitz functions}\label{sec_exp}

In this section we provide sufficient conditions for exponentially growing functions to be
integrable with respect to the invariant measure $\sigma$. This type of results are known for
various type of measures: for example the classical Fernique theorem (see e.g. \cite{Led}) says that
functions with exponential growth are integrable with respect to Gaussian measures in infinite
dimension, and similar results hold for discrete Bernoulli, or Poisson measures
(see \cite{BL98,Wu00}).

To get our results, we need to require further properties on the L\'{e}vy measure $M$.
\begin{hyp0}\label{hyp_exp}
For any $s \in (0,\infty)$ it holds that
\(M_s:=\int_X|y|^2 e^{s|y|}M(dy)\) is finite,
and there exist $C_0>0$, $\gamma\ge 1$ and $s_0>0$ such that for any $s\geq s_0$
\begin{equation}\label{psi}
\psi(s):=\int_{B_1^c}|y|e^{s|y|}M(dy)\ge C_0 e^{\gamma s}.
\end{equation}
\end{hyp0}

Note that the function $\psi:(0,\infty)\to (0,\infty)$ defined in \eqref{psi} is a continuous
nondecreasing function, hence its inverse is well-defined and it is continuous and nondecreasing, too. Moreover, we stress that, by \eqref{LevyM}, the
finiteness of $M_s$ in Hypothesis \ref{hyp_exp} is equivalent to the finiteness of the same
integral on $B_1^c$.

\begin{lemm}\label{pre-lem}
Under the hypotheses of Theorem \ref{thm_neve}, for any $f\in {\rm Lip}_b(X)$, with Lipschitz constant less than or equal to
$\tau$, it holds that
\begin{equation}\label{verme}
{\rm Ent}_\sigma (e^f)\le C\tau^2 M_{2\tau} m_\sigma(e^f),
\end{equation}
where $C$ is the constant appearing in \eqref{log_sob} and $M_{2\tau}$ is defined in Hypothesis \ref{hyp_exp}.
\end{lemm}
\begin{proof}
By Proposition \ref{Task} it suffices to prove the claim for
$f \in \mathcal{F}C^2_A(X)$ and use the dominated convergence theorem to complete the proof.

For $f\in\mathcal{F}C^2_A(X)$, the function $e^f$ belongs to $\mathcal{F}C^2_A(X)$ and has positive infimum.
Moreover, the mean value theorem together with the fact that $\|D f\|_\infty\le \tau$ yield
\begin{equation*}
|e^{f(x+y)}-e^{f(x)}|= e^{\theta}|f(x+y)-f(x)|\le \tau e^{\theta}|y|
\end{equation*}
for any $x,y \in X$ and some $\theta \in (f(x+y)\wedge f(x),f(x+y)\vee f(x))$. We can then
apply \eqref{log_sob} with $p=1$ to get
\begin{align*}
{\rm Ent}_\sigma (e^f) &\le C\int_X\int_X \frac{|e^{f(x+y)}-e^{f(x)}|^2}{e^{f(x)}}M(dy)\sigma(dx)
\\
& \le C\tau^2\int_X\int_X |y|^2\frac{e^{2\theta}}{e^{f(x)}}M(dy)\sigma(dx)
\\
& \le C\tau^2\int_X\int_X |y|^2e^{2(\theta-f(x))}e^{f(x)}M(dy)\sigma(dx)
\\
& \le C\tau^2\int_X\int_X |y|^2e^{2|f(x+y)-f(x)|}e^{f(x)}M(dy)\sigma(dx)
\\
& \le C\tau^2\int_X\int_X |y|^2e^{2\tau|y|}e^{f(x)}M(dy)\sigma(dx)
\\
&= C\tau^2 M_{2\tau}\int_Xe^{f(x)}\sigma(dx).\qedhere
\end{align*}
\end{proof}

The next result is a Fernique type theorem for the measure $\sigma$. The key tool is an
estimate of the tail of the distribution of a Lipschitz continuous function with respect to
$\sigma$ in terms of the function $\psi$ introduced in \eqref{psi} (cf. \cite{BL} for the Poisson case).

\begin{thm}\label{promosso}
Assume that the hypotheses of Theorem \ref{thm_neve} and Hypothesis \ref{hyp_exp} hold true. Any
Lipschitz continuous function $g:X \to \R$, with Lipschitz constant less than or equal to $1$, belongs to $L^1(X, \sigma)$ and
there exist positive constants $c_0,c_1,c_2$ and $t_0$ such that
\begin{equation}\label{G_tail}
\sigma \left(\{g\ge m_\sigma(g)+t\}\right)\le \left\{\begin{array}{ll}
\exp(-c_0t^2),\qquad\;\, &t \in (0, t_0);
\\
\exp(-c_1t \psi^{-1}(c_2t)),\qquad\;\, &t \in [t_0,\infty).
\end{array}
\right.
\end{equation}
Moreover, for sufficiently small $c>0$,
\begin{equation}\label{torta}
\int_X e^{c g\psi^{-1}(|g|)}d\sigma<\infty.
\end{equation}
\end{thm}
\begin{proof}
We divide the proof in three steps.

{\em Step 1.} We start considering $g\in {\rm Lip}_b(X)$, with Lipschitz constant less than or equal to $1$. The function
$\tau g$, with $\tau>0$ satisfies the assumptions of Lemma \ref{pre-lem} and consequently,
\begin{equation}\label{for1}
{\rm Ent}_\sigma (e^{\tau g})\le C\tau^2 M_{2\tau}m_{\sigma}(e^{\tau g}).
\end{equation}
Moreover, the function $G(\tau):=m_{\sigma}(e^{\tau g})$ is differentiable and, from \eqref{for1}, its derivative $G'(\tau)=m_{\sigma}(ge^{\tau g})$ satisfies
\begin{align}\label{fat}
\tau G'(\tau)-G(\tau)\log G(\tau)\le C\tau^2 M_{2\tau}G(\tau),\qquad \tau>0.
\end{align}
Thus, set
\[
H(\tau):=\eqsys{ m_\sigma(g), & \tau=0;\\ \tau^{-1}\log G(\tau), & \tau>0.}
\]
By \eqref{fat} we deduce that $H'(\tau)\le C  M_{2\tau}$ whence, integrating from $0$
to $t$, we have
\begin{align}\label{crucial}
H(t)-H(0)&\le C\int_0^t \int_X |y|^2 e^{2\tau |y|}M(dy)d\tau
\notag\\
&=  C\int_X |y|^2\int_0^t e^{2\tau |y|}d\tau M(dy)
\notag\\
& = \frac{C}{2}\int_X |y|(e^{2t |y|}-1) M(dy)=:\theta(t)
\end{align}
or, equivalently, $m_\sigma(e^{tg})\le \exp\Big(t(\theta(t)+m_\sigma(g))\Big)$.
Applying the Chebyshev inequality we get
$$
\sigma\Big(\{g \ge m_\sigma(g)+s\}\Big)\le \exp\Big(-ts+\frac{C t}{2}\int_{X}|y|(e^{2t|y|}-1)M(dy)\Big).
$$
Now, using the inequality $e^{\alpha x}-1\le (e^{\alpha}-1)x$ for any $\alpha>0$ and $x \in (0,1)$,
we can estimate
\begin{align*}
\sigma\Big(\{g \ge m_\sigma(g)+s\}\Big)&\le \exp\Big(-ts+\frac{C t}{2}(e^{2t}-1)
\int_{B_1}|y|^2M(dy)+\frac{C t}{2}\int_{B_1^c}|y|(e^{2t|y|}-1)M(dy)\Big)
\\
&=\exp\Big( -ts+C_1 t(e^{2t}-1)+C_2t\int_{B_1^c}|y|(e^{2t|y|}-1)M(dy)\Big)
=:\exp(\varphi(t,s))
\end{align*}
for any $t,s>0$.
Let us fix $0<\alpha < (C_1C_0^{-1}e^{s_0(1-\gamma)}+C_2)^{-1}$ and distinguish two cases. \\
As the first one we take
$s \ge C_0\alpha^{-1}e^{\gamma s_0}$, where $s_0>0$ (see Hypothesis \ref{hyp_exp}) is such that
$$
\psi(\tau) \ge C_0 e^{\gamma\tau}, \qquad\,\, \tau \ge s_0.
$$
In such case, we choose $t=2^{-1} \psi^{-1}(\alpha s)$ and get
\begin{align}\label{mom}
\varphi(2^{-1}\psi^{-1}(\alpha s),s)=&
-2^{-1}s\psi^{-1}(\alpha s)+C_12^{-1}\psi^{-1}(\alpha s)(e^{\psi^{-1}(\alpha s)}-1)
\notag\\
&+C_22^{-1}\psi^{-1}(\alpha s)\int_{B_1^c}|y|(e^{\psi^{-1}(\alpha s)|y|}-1)M(dy)
\notag\\
\le&  -2^{-1}s\psi^{-1}(\alpha s)+C_12^{-1}\psi^{-1}(\alpha s)e^{\psi^{-1}(\alpha s)}+
C_22^{-1}\alpha s\psi^{-1}(\alpha s)
\notag\\
=& -2^{-1}s\psi^{-1}(\alpha s)\Big( 1-C_1s^{-1}e^{\psi^{-1}(\alpha s)}-C_2\alpha \Big).
\end{align}
Using Hypothesis \ref{hyp_exp} we deduce that $e^{\psi^{-1}(z)}\le (z C_0^{-1})^{1/\gamma}$ for any
$z\ge C_0 e^{\gamma s_0}$. Applying the last estimate with $z= \alpha s$ in \eqref{mom}, we
conclude that $\varphi(2^{-1}\psi^{-1}(\alpha s),s)\le -c_1s\psi^{-1}(\alpha s)$ for some positive $c_1$. \\
As second case, if $s \le  C_0\alpha^{-1}e^{\gamma s_0}=:s_1 $, choosing $t=\beta s$ with a suitable
$\beta \in (0,1)$ we can show that $\varphi (\beta s, s)\le -\bar{\beta}s^2$ for some $\bar{\beta}>0.$
Indeed, using the estimate
$$
e^{\eta s}-1\le s \frac{e^{\eta s_1}-1}{s_1}, \qquad s \in (0,s_1],\, \eta >0,
$$
we have
\begin{align}\label{estimate}
\varphi(\beta s, s)&= -\beta s^2+C_1 \beta s(e^{2\beta s}-1)+
C_2\beta s\int_{B_1^c}|y|(e^{2\beta s|y|}-1)M(dy)
\notag\\
& \le -\beta s^2+C_1 \beta s^2 s_1^{-1}(e^{2\beta s_1}-1)+
C_2\beta s^2s_1^{-1}\int_{B_1^c}|y|(e^{2\beta|y|s_1}-1)M(dy)
\notag\\
&= -\beta s^2\Big(1-C_1 s_1^{-1}(e^{2\beta s_1}-1)-C_2s_1^{-1}\int_{B_1^c}|y|(e^{2\beta|y|s_1}-1)M(dy)\Big).
\end{align}
Now, if $\beta \in (0,1)$ is such that
\begin{equation*}
\beta \le s_1\Big[C_1e^{2s_1}+C_2\int_{B_1^c}|y|e^{2s_1|y|}M(dy)\Big]^{-1}
\end{equation*}
then the term in round brackets in the right hand side of \eqref{estimate} is positive and
$\varphi (\beta s, s)\le -\bar{\beta}s^2$ for some $\bar{\beta}>0$.
Summing up, estimate \eqref{G_tail} follows.

{\em Step 2.} Here we consider the general case and we approximate any Lipschitz continuous function $g$, with Lipschitz constant less than or equal to $1$, by the sequence of functions $(g_n)_n$ defined by $g_n:=(-n)\wedge(g\vee n)$, $n\in \N$ which converges pointwise to $g$. Then, applying the previous step to $-|g_n|$ we infer that
$$\sigma \left(\{|g_n|\le m_\sigma(|g_n|)-t\}\right)\le
\exp(-c_1t\psi^{-1}(c_2t))$$
for $t$ large enough. Choosing $t_0$ such that $\exp(-c_1t_0\psi^{-1}(c_2t_0))\le 1/2$ and $m$ such that
$\sigma (\{|g|\ge m)\}<1/2$ and using that $|g_n|\le |g|$ we deduce that
$\|g_n\|_{L^1(X,\sigma)}=m_\sigma(|g_n|)\le m+t_0$ for any $n \in \N$. Indeed, by contradiction,
if  $m_\sigma(|g_n|)>m+t_0$, we have
\begin{align*}
\sigma(\{|g|<m\})&\le \sigma(\{|g_n|<m\}) = \sigma(\{|g_n|+t_0<m+t_0\})
\\
&\le \sigma(\{|g_n|+t_0<m_\sigma(|g_n|)\})
= \sigma(\{|g_n|<m_\sigma(|g_n|)-t_0\})\le 1/2
\end{align*}
which yields a contradiction with $\sigma (\{|g|\ge m\})<1/2$, as $\sigma$ is a probability measure.
Hence, from the previous estimate  and the monotone convergence theorem we get that
$g\in L^p(X, \sigma)$ for any $p \ge 1$ and that $\|g_n\|_{L^p(X,\sigma)}$ converges to
$\|g\|_{L^p(X,\sigma)}$ as $n \to \infty$. Moreover, using \eqref{G_tail} and that
$m_\sigma(|g_n|)\le m+t_0$ we obtain
\begin{equation*}
\sigma(\{|g_n|\ge m+t_0+t\})\le \exp(-c_1t\psi^{-1}(c_2t))\le \exp(-c_1 t_0\psi^{-1}(c_2t_0))
\end{equation*}
for $t\ge t_0$ whence $\sup_n\|g_n\|_{L^2(X, \sigma)}<\infty$. By a standard compactness argument
we get that $g_n$ converges to $g$ in $L^2(X, \sigma)$ as $n \to \infty$ and hence in measure,
i.e., for any $\varepsilon>0$
\begin{equation}\label{measure}
\lim_{n \to \infty}\sigma(\{|g_n-g|\ge \varepsilon\})=0.
\end{equation}
From \eqref{measure} and \eqref{G_tail} for $g_n$ we infer that \eqref{G_tail} holds true for $g$ too.

{\em Step 3.} Here we prove the last statement. Let $c>0$ and
$\varphi_c(x):=cx\psi^{-1}(|x|)$. Observing that $\varphi_c$ is a nondecreasing
and invertible function, we have
\begin{align*}
\int_X e^{c g \psi^{-1}(|g|)} d\sigma &= \int_X e^{\varphi_c(g)} d\sigma
\\
&\le 1+\int_1^\infty \sigma (\{e^{\varphi_c(g)}>s\})ds
\\
&=1+\int_1^\infty \sigma (\{\varphi_c(g)>\log s\})ds
\\
&=1+\int_1^\infty \sigma (\{g>\varphi_c^{-1}(\log s)\})ds
\\
&\leq K_1+c\int_{T}^\infty \sigma (\{g>t\})e^{\varphi_c(t)}(K_2+\log t)dt,
\end{align*}
for some $K_1,K_2>0$ and $T:=\varphi_c^{-1}(0)\vee (C_0 e^{\gamma s_0})$. Now, performing
the change of variable $\tau=t-m_\sigma(g)$ and using
the estimate \eqref{G_tail} we have
\begin{align*}
\int_{T}^\infty &\sigma (\{g>t\})e^{\varphi_c(t)}(K_2+\log t)dt
\\
&=\int_{T-m_\sigma(g)}^\infty \sigma (\{g>\tau+m_\sigma(g)\})
e^{\varphi_c(\tau+m_\sigma(g))}(K_2+\log(\tau+m_\sigma(g))d\tau
\\
&\le C+\int_{t_0}^\infty \sigma (\{g>\tau+m_\sigma(g)\})
e^{\varphi_c(\tau+m_\sigma(g))}(K_2+\log(\tau+m_\sigma(g))d\tau
\\
& \le C+\int_{t_0}^\infty e^{-c_1\tau\psi^{-1}(c_2 \tau)
+c(\tau+m_\sigma(g))\psi^{-1}(|\tau+m_\sigma(g)|)}(K_2+\log(\tau+m_\sigma(g))d\tau,
\end{align*}
for some positive $C$. Then, for $c$ small enough the function $e^{c g \psi^{-1}(|g|)}$
is summable and the proof is complete.
\end{proof}

\begin{rmk}{\rm
Under the hypotheses of Theorem \ref{promosso} and assuming $\psi(s)= C_0e^{\gamma s}$
for some $C_0>0$, $\gamma \ge 1$ and $s$ big enough, we have
$$
\sigma \left(\{g\ge m_\sigma(g)+t\}\right)\le \exp(-c_1t\log(c_2t))
$$
for some positive $c_1,c_2$ and $t$ large enough.
Viceversa, if Hypothesis \ref{hyp_exp} holds true with a strict inequality in \eqref{psi} then
there exist $K<1$ such that $\psi^{-1}(t) \le  K \log t$ as $t \to \infty$. If in Hypothesis \ref{hyp_exp} the estimate \eqref{psi} holds true for any $\gamma \ge 1$, then we
conclude that $\psi^{-1}(t)=o(\log t)$ as $t \to \infty$.}
\end{rmk}

\begin{cor}
Under the hypotheses of Theorem \ref{promosso} it follows that $\sigma(p)<\infty$ for any $p\ge 1$.
\end{cor}
\begin{proof}
It suffices to take $g(x)=|x|$ in \eqref{torta} and observe that there exist $r_0\in (0,\infty)$
such that $\psi^{-1}(r_0)>0$ and $cg(x)\psi^{-1}(g(x))\ge c\psi^{-1}(r_0)|x|$ for any $|x| \ge r_0$.
\end{proof}

\begin{rmk}\label{gauss_exp}\rm{
Let us show how the results in this section can be reformulated if the Gaussian term $Q$
in the representation of $\lambda\leftrightarrow [b,Q,M]$ does not vanish and satisfies
the assumptions in Remark \ref{gaussiana}. We recall that, in this case, under the
hypotheses of Theorem \ref{promosso} there is a unique invariant measure associated with $P_t$,
given by $\gamma * \sigma$, where $\gamma$ is the Gaussian measure defined in Remark \ref{gaussiana}.
Assuming further that estimate \eqref{gra-est} holds true, let us
reformulate the results in Theorem \ref{promosso}. Indeed, we can prove that any Lipschitz
continuous function $g:X \to \R$, with Lipschitz constant less than or equal to $1$, belongs to $L^1(X, \gamma*\sigma)$ and
there exist positive constants $\tilde{c}_0,\tilde{c}_1,\tilde{c}_2$ and $\tilde{t}_0$ such that
\begin{equation}\label{stima_g1}
(\gamma*\sigma) \left(\{g\ge m_\sigma(g)+t\}\right)\le \left\{\begin{array}{ll}
\exp(-\tilde{c}_0t^2),\qquad\;\, &t \in (0, \tilde{t}_0);
\\
\exp(-\tilde{c}_1t \psi^{-1}(\tilde{c}_2t)),\qquad\;\, &t \in [\tilde{t}_0,\infty).
\end{array}
\right.
\end{equation}
Moreover, for sufficiently small $\tilde{c}>0$,
\begin{equation}\label{stima_g2}
\int_X e^{\tilde{c} g\psi^{-1}(|g|)}d(\gamma*\sigma)<\infty.
\end{equation}
To prove this fact, it suffices to perform some changes in Step 1 in the proof of Theorem \ref{promosso},
as described below. First of all, notice that $m_\gamma (f)$ and $m_\sigma (f)$ can be estimated
by $m_{\gamma*\sigma}(f)$ for any nonnegative function $f$. Thanks to estimate \eqref{bing},
formula \eqref{verme} becomes
\begin{equation}\label{verme1}
{\rm Ent}_{\gamma*\sigma}(e^f)\le K \tau^2 (M_{2\tau}+1)m_{\gamma *\sigma}(e^f)
\end{equation}
for any $f \in C^1_b(X)$ with $\|Df\|_\infty \le \tau$ and some positive constant $K$ depending
also by $\|Q\|_{\mathcal{L}(X)}$.
Then, considering $\tilde{G}(\tau):=m_{\gamma*\sigma}(e^{\tau g})$ and
\[
\tilde{H}(\tau):=\eqsys{ m_{\gamma *\sigma}(g), & \tau=0,\\ \tau^{-1}\log \tilde{G}(\tau), & \tau>0}
\]
in place of $G$ and $H$ (see Step 1 in the proof of Theorem \ref{promosso}) and using \eqref{verme1}
we deduce that $\tilde{H}'(\tau)\le  K(M_{2\tau}+1)$, whence integrating from $0$ to $t$ we have
\begin{equation*}
\tilde{H}(t)-\tilde{H}(0)\le K (\theta(t)+t)
\end{equation*}
where, up to a constant, $\theta$ is defined by
$\theta(t):= \int_X |y|(e^{2t |y|}-1) M(dy)$.
At this point, the proof could be repeated slavishly if we prove that there exists a positive constant $c$ such that $t \le c\theta(t)$ for any $t>0$. Indeed, in such case we would have $\tilde{H}(t)-\tilde{H}(0)\le K'\theta(t)$ for some positive $K'$ as in \eqref{crucial}.
The estimate $t \le c\theta(t)$ can be proved using that $e^x-1\ge x$ for any $x \ge 0$ and observing that
\begin{equation*}
\theta(t)=\int_X |y|(e^{2t |y|}-1) M(dy)\ge 2t \int_X |y|^2M(dy)\ge 2t M_s
\end{equation*}
for any $s\in (0,\infty)$, see Hypothesis \ref{hyp_exp}. In this way all the results stated for $\sigma$ in Theorem \ref{promosso} are true for $\gamma * \sigma$ too. In particular, \eqref{stima_g1} and
\eqref{stima_g2} follow.
}\end{rmk}

\section{Examples}\label{SecEx}

Here we provide examples of generalised Mehler semigroups satisfying our assumptions and to which
our results can be applied.

\subsection{Ornstein--Uhlenbeck operators with fractional diffusion}\label{OUexample}
We consider the Ornstein--Uhlenbeck operator defined by
$$
(\mathcal{L}u)(x)=\frac{1}{2}[{\rm Tr}^s (QD^2u)](x)
+\langle Bx, Du(x)\rangle,\qquad\,\; s\in (0,1),\ x\in \Rd .
$$
We assume that $Q$ is a symmetric nonnegative definite matrix, $B$ is a symmetric nonpositive definite matrix and
${\rm Tr}^s (QD^2u)$ is the pseudo-differential operator with symbol $\langle Q\xi,\xi\rangle^s$.
The realisation of $\mathcal{L}$ in $L^2(\Rd)$ has been studied in \cite{AlBer} and in the space of
H\"older continuous functions in \cite{LR}. The associated generalised Mehler semigroup is given by
\begin{align*}
(P_tf)(x)= \int_{\Rd} f(e^{t B}x+y)\mu_t(dy),
\end{align*}
where
$$
\hat{\mu}_t(\xi)= \exp \Bigg(-\int_0^t \lambda (e^{\tau B^*}\xi)d\tau\Bigg)=
\exp\left(-\frac{1}{2}\int_0^t |Q^{1/2}e^{\tau B^*}\xi|^{2s}d\tau\right),\qquad\;\, \xi \in \Rd.
$$
If $Q$ is invertible, then the L\'evy measure $M$ defined in \eqref{lambda} has the form
\begin{align*}
M(A)=\frac{1}{\det Q^{1/2}}\int_{\R^d}\frac{\chi_A(y)}{|Q^{-1/2}y|^{2s+d}}dy,
\qquad\; \, A\in \mathcal{B}(\Rd).
\end{align*}
Moreover, from \cite[Section 7]{FR}, Hypotheses \ref{base} are satisfied and consequently there exists
a unique invariant measure $\sigma$ for $P_t$. Such measure is absolutely continuous with respect to the
Lebesgue measure with density $\rho$ satisfying
\begin{equation*}
\frac{1}{C(1+|x|^{d+2s})}\le \rho(x)\le \frac{C}{1+|x|^{d+2s}}
\end{equation*}
for some $C > 1$ and any $x \in \Rd$ (see \cite{BG60}). It is also clear that the measure
$M$ satisfies \eqref{App_cond} whenever $s \in \left(\frac{1}{2},1\right)$. From now on we restrict
our attention to this case and prove that, under further conditions,  Hypotheses \ref{hyp_wang}
are satisfied, too, as the next proposition shows.

\begin{prop}\label{h2_OU}
Assume that $Q$ and $B$ are as above. If $Q$ and $B$ are invertible and commute, then Hypotheses
\ref{hyp_wang} hold true whenever
\begin{equation}\label{auto}
\min \set{|\lambda_i|\,|\,i=1,\ldots,d} > \frac{|{\rm Tr}B|}{2s+d},
\end{equation}
where the $\lambda_i$ are the eigenvalues of $B$.
\end{prop}
\begin{proof}
Let $\Lambda:=\max \set{\lambda_i\,|\,i=1,\ldots,d}<0$. We have
\begin{align*}
[M\circ e^{-tB}](A)=&\frac{1}{\det Q^{1/2}}\int_{\R^d}\frac{\chi_A(e^{tB}y)}{|Q^{-1/2}y|^{2s+d}}dy
\\
=&\frac{e^{-t{\rm Tr}B}}{\det Q^{1/2}}\int_{\R^d}\frac{\chi_A(z)}{|Q^{-1/2}e^{-tB}z|^{2s+d}}dz
\\
=&\frac{e^{-t{\rm Tr}B}}{\det Q^{1/2}}\int_{\R^d}\frac{\chi_A(z)}{|e^{-tB}Q^{-1/2}z|^{2s+d}}dz
\\
\leq&\frac{e^{-t{\rm Tr}B}}{e^{-(2s+d)\Lambda t}\det Q^{1/2}}
\int_{\R^d}\frac{\chi_A(z)}{|Q^{-1/2}z|^{2s+d}}dz.
\end{align*}
So the function $h $ in Hypotheses \ref{hyp_wang} is
\[
h (t):=e^{-t({\rm Tr}B-(2s+d)\Lambda)},
\]
that clearly belongs to $L^1((0,\infty))$ whenever $|\Lambda|> \frac{|{\rm Tr}B|}{2s+d}$.
\end{proof}

As examples of matrices satisfying the hypotheses of Proposition \ref{h2_OU}, one can
consider as $Q$ a positive definite invertible matrix and as $B$ one of the following:
\begin{enumerate}[\rm (a)]
\item $B=-\beta I$ for any $\beta>0$;
\item $B=-Q^\alpha$ for some $\alpha>0$.
In this case, condition \eqref{auto} becomes
\begin{align*}
r_1>\pa{\frac{1}{2s+d-1}\sum_{i=2}^d r_i^\alpha}^{1/\alpha},
\end{align*}
where $r_1$ is the minimum eigenvalue of $Q$.
\end{enumerate}

Under the previous conditions all the assumptions in Theorem \ref{thm_neve} and Propositions
\ref{Poincare}, \ref{ciao} are satisfied. Then the logarithmic Sobolev inequality
$$
{\rm Ent}_\sigma(|f|)\leq K_{1,f}\int_{\R^d}\int_{\R^d}\frac{|f(x+y)-f(x)|^2}{|y|^{d+2s}} dy\sigma(dx)
$$
holds true for any $f\in \mathcal{H}^2_0$ and some positive $K_{1,f}$. The Poincar\'e inequality
$$
\|f-m_\sigma(f)\|_{L^2(X, \sigma)}^2\le K_2\int_{\R^d} \int_{\R^d} \frac{|f(x+y)-f(x)|^2}{|y|^{d+2s}}dy\sigma(dx)
$$
holds as well for any $f \in \mathcal{H}^2$ and some positive $K_2$.

\subsection{An example in infinite dimension}

Let $X$ be a separable Hilbert space and let $\lambda\leftrightarrow [b,0,M]$ where $b\in X$ and $M$ is an
infinitely divisible $\alpha$-stable L\'{e}vy measure. Consider the semigroup $(T_t)_{t\geq 0}$
defined by $T_tx:=e^{-t\beta}x$ for some $\beta>0$ and every $x\in X$ and $t\geq0$.

We recall that an infinitely divisible L\'{e}vy measure is $\alpha$-stable with $\alpha\in(0,2)$
if and only if there exists a finite measure $\mu$ concentrated on the unit sphere of $X$ such that for any Borel set $B\subseteq X$
\[
M(B)=\int_0^{\infty}r^{-1-\alpha}\pa{\int_{S_1}\chi_B(rx)\mu(dx)}dr,
\]
where $S_1=\{x\in X\,|\,|x|=1\}$ denotes the unit sphere of $X$, see \cite[Theorem 6.2.8]{Lin86} for more details.
Let us show that if $\alpha \in (1,2)$ and $\mu$ is a symmetric measure then Hypothesis \ref{hyp_App},
\ref{base} and \ref{hyp_wang} are all satisfied.

Indeed, observing that
\begin{align*}
\int_{B_1^c}|x|M(dx)=\int_1^{\infty}r^{-1-\alpha}\pa{\int_{S_1}|rx|\mu(dx)}dr=
\mu(S_1)\int_1^{\infty}r^{-\alpha}dr<\infty,
\end{align*}
we deduce that Hypothesis \ref{hyp_App} holds true. Now observe that
\begin{align*}
\lim_{t\ra \infty}b_t&=\lim_{t\ra \infty}
\pa{\int_0^tT_rbdr+\int_0^t\int_XT_rx(\chi_{B_1}(T_rx)-\chi_{B_1}(x))M(dx)dr}
\\
&=\lim_{t\ra \infty}\pa{\int_0^te^{-r\beta}bdr+
\int_0^t\int_0^{\infty}s^{-1-\alpha}\int_{S_1}e^{-r\beta}sy(\chi_{B_1}(e^{-r\beta}sy)-
\chi_{B_1}(sy))\mu(dy)ds dr}
\\
&=\lim_{t\ra \infty}\Bigg[\frac{b}{\beta}(1-e^{-t\beta})+
\int_0^te^{-r\beta}\Bigg(\int_0^{e^{r\beta}}s^{-\alpha}ds+
\int_0^1 s^{-\alpha}ds\Bigg)dr\int_{S_1}y\mu(dy)\Bigg]=\frac{b}{\beta},
\end{align*}
where in the last equality we used the simmetry of $\mu$. Furthermore, arguing as in
\eqref{est_M} we deduce
\begin{align*}
\int_0^{\infty}\int_X(1\wedge |T_rx|^2)M(dx)dr&=
\lim_{t \to \infty}\int_0^{t}\int_X(1\wedge |T_rx|^2)M(dx)dr
\\
& \le \frac{1}{2\beta}\int_X(1\wedge |x|^2)M(dx)<\infty ,
\end{align*}
hence Hypothesis \ref{base}(ii) holds true. Clearly Hypothesis \ref{base}(iii) holds true too and
Theorem \ref{furhman} can be applied to prove the existence of a unique invariant measure
$\sigma$ for $P_t$.  Finally we show that Hypotheses \ref{hyp_wang} are satisfied as well.
For every $A\in \mathcal{B}(X)$ we have
\begin{align*}
[M\circ T^{-1}_t](A)&=\int_0^{\infty}\int_{S_1}\chi_{T_t^{-1}(A)}(rx)\mu(dx)r^{-1-\alpha}dr
\\
&=\int_0^{\infty}\int_{S_1}\chi_{A}(re^{-t\beta}x)\mu(dx)r^{-1-\alpha}dr
\\
&=e^{-\alpha t\beta}\int_0^{\infty}\int_{S_1}\chi_{A}(\rho x)\mu(dx)\rho^{-1-\alpha}d\rho
\\
&=e^{-\alpha t\beta}M(A).
\end{align*}
Hence Hypotheses \ref{hyp_wang} hold true with $h (t)=e^{-\alpha t\beta}$ that clearly belongs
to $L^1((0,\infty))$. Hence all the results in Section \ref{sec_log} can be applied and,
in particular, Theorem \ref{thm_neve} states that for every $f\in \mathcal{F}C^2_{A}(X)$ with
positive infimum and $p\in[1,\infty)$
\begin{align*}
{\rm Ent}_\sigma(f^p)\leq
C\int_X\int_X\frac{|f^p(x+y)-f^p(x)|^2}{f^p(x)} M(dy)\sigma(dx).
\end{align*}

\subsection{A modified version of a model introduced by Koponen}

A tempered stable process is obtained from a one dimensional stable process by ``tempering'' the large jumps,
i.e., by damping exponentially the tails of the L\'evy measure. This class of L\'{e}vy processes has been
introduced by Koponen in \cite{Kop} for options pricing (see also \cite[Section 13.4.3]{Pas11}).

Here we slightly modify the process introduced by Koponen providing a L\'{e}vy measure $M$ and consequently
a generalised Mehler semigroup $P_t$ which satisfy all our assumptions.
For semplicity we consider the one dimensional and centred case but it is not difficult to extend
the results for any dimension and in the non-centred case.

Let $X=\R$ and consider the semigroup $T_tx:=e^{-t\beta}x$ for some $\beta>0$. The L\'{e}vy process we
are interested in is identified by the triple $[b,0,M]$ where $b\in \R$ and
\begin{align*}
M(dx):=c\frac{e^{-{x^2}}}{|x|^{1+{2s}}}dx,
\end{align*}
for some $c>0$ and $s\in(0,1)$.
First of all we show that $P_t$ admits an invariant measure. Indeed, recalling formula
\eqref{scrivere meglio}  we have
\begin{align*}
\lim_{t\ra \infty}b_t &=\lim_{t\ra \infty}\pa{\int_0^t T_r b dr+\int_0^t\int_X T_r x\Big(\chi_{B_1}(T_r x)-\chi_{B_1}(x)\Big)M(dx)dr}\\
&=\lim_{t\ra\infty}\pa{\frac{b}{\beta}(1-e^{-t\beta})+c\int_0^te^{-r\beta}\int_{-e^{r\beta}}^{e^{r\beta}}x\frac{e^{-x^2}}{|x|^{1+2s}}dxdr-c\int_0^te^{-r\beta}\int_{-1}^{1}x\frac{e^{-x^2}}{|x|^{1+2s}}dxdr}\\
&=\lim_{t\ra\infty}\frac{b}{\beta}(1-e^{-t\beta})=\frac{b}{\beta}.
\end{align*}
Moreover
\begin{align*}
\int_0^{\infty}\int_\R(1\wedge |T_rx|^2)M(dx)dr &=c\int_0^{\infty}
\int_\R(1\wedge |e^{-r\beta}x|^2)\frac{e^{-x^2}}{|x|^{1+2s}}dxdr
\\
&=2c\int_0^{\infty}\int_{e^{r\beta}}^{\infty}\frac{e^{-x^2}}{x^{1+2s}}dxdr
+2c\int_0^{\infty}e^{-2r\beta}\int_{0}^{e^{r\beta}} x^2\frac{e^{-x^2}}{x^{1+2s}}dxdr
\\
&\leq 2c\int_0^{\infty}\int_{e^{r\beta}}^{\infty}\frac{1}{x^{1+2s}}dxdr
+2c\int_0^{\infty}e^{-2r\beta}\int_{0}^{e^{r\beta}} \frac{1}{x^{2s-1}}dxdr
\\
&= \frac{c}{2s^2(1-s)\beta}.
\end{align*}
So, Theorem \ref{furhman} applies and a unique invariant measure $\sigma$ exists for $P_t$. Furthermore, using the estimate $x^2e^{-x^2}\leq e^{-1}$, $x\in\R$, we get
\begin{align*}
\int_{B_1^c}|x|M(dx)=2c\int_1^{\infty}x\frac{e^{-x^2}}{x^{1+2s}}dx\leq
\frac{2c}{e}\int_1^{\infty}\frac{1}{x^{2s+2}}dx=\frac{2c}{(2s+1)e},
\end{align*}
whence Hypothesis \ref{hyp_App} hold true and $C_A^2(\R)$ is a core for the generator of the semigroup
$P_t$ in $L^2(\R,\sigma)$.

In order to apply the results in Sections \ref{sec_log} and \ref{sec_exp} we have to verify Hypotheses
\ref{hyp_wang} and \ref{hyp_exp}, respectively. Let us start from Hypotheses \ref{hyp_wang}. If $A$ be a Borel subset of $\R$, then
\begin{align*}
[M\circ T_t^{-1}](A)&=c\int_\R\chi_{T_t^{-1}(A)}(x)\frac{e^{-x^2}}{|x|^{1+2s}}dx=ce^{t\beta}
\int_\R\chi_{T_t^{-1}(A)}(e^{t\beta}y)\frac{e^{-e^{2\beta t}y^2}}{|e^{\beta t}y|^{1+2s}}dy
\\
&=ce^{-2st\beta}\int_\R\chi_{A}(y)\frac{e^{-e^{2\beta t}y^2}}{|y|^{1+2s}}dy=
ce^{-2st\beta}\int_\R\chi_{A}(y)e^{(1-e^{2\beta t})y^2}\frac{e^{-y^2}}{|y|^{1+2s}}dy
\\
&\leq ce^{-2st\beta}\int_\R\chi_{A}(y)\frac{e^{-y^2}}{|y|^{1+2s}}dy=e^{-2st\beta}M(A).
\end{align*}
Now, the function $h (t):=e^{-2st\beta}$ is continuous in $(0,\infty)$ and belongs to $L^1((0,\infty))$,
hence all the results in Section \ref{sec_log} can be applied and, in particular, Theorem \ref{thm_neve}
states that for every $f\in C_b(\R)$ with positive infimum and $p\in[1,\infty)$
\begin{align*}
{\rm Ent}_\sigma(f^p)\leq
C\int_\R\int_\R\frac{|f^p(x+y)-f^p(x)|^2}{f^p(x)} \frac{e^{-y^2}}{|y|^{1+2s}}dy\sigma(dx).
\end{align*}
To conclude, let us consider Hypothesis \ref{hyp_exp}. Indeed recalling that for any $\alpha>0$ there
exists $K(\alpha)>0$ such that $y^4e^{\alpha y-y^2}\leq K(\alpha)$ for any $y>0$ we have
\begin{align*}
\int_{B_1^c}|y|^2e^{\alpha |y|}M(dy)&=2c\int_1^{\infty}y^2e^{\alpha y}\frac{e^{-y^2}}{y^{1+2s}}dy
\\
&\leq 2c\int_1^{\infty}y^2e^{\alpha y-y^2}dy
\\
&\leq 2cK(\alpha)\int_1^{\infty}\frac{1}{y^2}dy=\frac{2c}{3}K(\alpha),
\end{align*}
and, again, for $\alpha\in(0,\infty)$
\begin{align*}
\psi(\alpha)&=\int_{B_1^c}|y|e^{\alpha|y|}M(dy)=
2c\int_1^{\infty}ye^{\alpha y}\frac{e^{-y^2}}{y^{1+2s}}dy
\ge 2ce^{\alpha}\int_1^{\infty}\frac{e^{-y^2}}{y^{2s}}dy=
Ce^{\alpha}.
\end{align*}
Thus, all the results in Section \ref{sec_exp} can be applied and, in particular, Theorem \ref{promosso} guarantees the exponential integrability of Lipschitz continuous functions with respect to $\sigma$.

\begin{rmk}\label{gauss_ex}{\rm
All the examples considered above can be modified adding a Gaussian term satisfying the assumptions in Remark \ref{gaussiana}. In such case, all the estimates and the results stated for $\sigma$ can be reformulated for $\gamma *\sigma$, see Remarks \ref{gaussiana} and \ref{gauss_exp} for a detailed descriptions of the results.}
\end{rmk}

%
%
%

\appendix

\section{Approximation of and by Lipschitz functions}
We collect here, for the reader's convenience, some results we used in the paper, even though
their proofs are quite standard.

\begin{lemm}\label{lusin BUC}
Let $X$ be a metric space, $\theta$ be a finite Radon measure on $X$ and $f:X\ra\R$ be a Borel function.
For every $\eps>0$ there exists a bounded uniformly continuous function $g_\eps:X\ra\R$ such that
\[
\theta\pa{\set{x\in X\tc f(x)\neq g_\eps(x)}}<\eps.
\]
Furthermore $\norm{g_\eps}_\infty\leq 2\norm{f}_\infty$.
\end{lemm}
\begin{proof}
Consider a compact set $K_0\subseteq X$ such that $\theta(X\smallsetminus K_0)<\eps/2$.
The function $f_{|_{K_0}}:K_0\ra\R$ is a Borel function and by Lusin theorem
(see \cite[Theorem 2.24]{Rud87}) there exists a continuous function $\tilde{g}_\eps:K_0\ra\R$ such that
\[
\theta\pa{\set{x\in K_0\tc f_{|_{K_0}}(x)\neq \tilde{g}_\eps(x)}}<\frac{\eps}{2}
\]
and
\[
\sup_{x\in K_0}|\tilde{g}_\eps(x)|\leq\sup_{x\in K_0}\abs{f_{|_{K_0}}(x)}=\sup_{x\in K_0}\abs{f(x)}.
\]
The Heine--Cantor theorem says that $\tilde{g}_\eps$ is a bounded and uniformly continuous function on $K_0$. Consider the bounded and uniformly continuous extension
(see \cite{Man90})
\[
g_\eps(x)=\eqsys{\tilde{g}_\eps(x) & x\in K_0;\\ \inf_{y\in K_0}\tilde{g}_\eps(y)\frac{d(x,y)}{{\rm dist}(x,K_0)} & x\notin K_0.}
\]
An easy computation gives that for every $x\notin K_0$
\[
\abs{g_\eps(x)}\leq \sup_{z\in X}\abs{f(z)}.
\]
Eventually we get
\begin{gather*}
\sup_{x\in X}\abs{g_\eps(x)}\leq\sup_{x\in K_0}\abs{g_\eps(x)}
+\sup_{x\in X\smallsetminus K_0}\abs{g_\eps(x)}
\leq \sup_{x\in K_0}\abs{\tilde{g}_\eps(x)}+\sup_{x\in X\smallsetminus K_0}\abs{g_\eps(x)}
\leq 2\sup_{x\in X}\abs{f(x)}.
\end{gather*}
Furthermore
\begin{align*}
\theta\pa{\set{x\in X\tc f(x)\neq g_\eps(x)}}&\leq
\theta(X\smallsetminus K_0)+\theta\pa{\set{x\in K_0\tc f(x)\neq g_\eps(x)}}
\\
&= \theta(X\smallsetminus K_0)+\theta\pa{\set{x\in K_0\tc f_{|_{K_0}}(x)\neq \tilde{g}_\eps(x)}}<\eps.
\qedhere
\end{align*}
\end{proof}

\begin{prop}\label{approx}
Let $X$ be a metric space, let $\theta$ be a finite Radon measure on $X$ and $p\geq 1$.
The space ${\rm Lip}_b(X)$ of bounded and Lipschitz continuous functions on $X$ is dense in $L^p(X,\theta)$.
\end{prop}
\begin{proof}
Fix a version of a $f\in L^p(X,\theta)$. For every $k\in\N$ set
\[
f_k(x)=\eqsys{k & f(x)>k;\\ f(x) & -k\leq f(x)\leq k;\\ -k & f(x)<-k.}
\]
Applying Lemma \ref{lusin BUC}, for every $k\in\N$ there is a bounded and uniformly continuous function $\tilde{f}_k$ such that
\[
\theta\pa{\set{x\in X\tc \tilde{f}_k(x)\neq f_k(x)}}\leq \frac{1}{2^k},
\]
and $\sup_{x\in X}|\tilde{f}_k(x)|\leq 2\sup_{x\in X}\abs{f_k(x)}\leq 2k$. Theorem 1 in \cite{Mic02}
gives a function $g_k\in\lip_b(X)$ such that
\[
\|g_k-\tilde{f}_k\|_\infty\leq\frac{1}{2^k}.
\]
We have
\[
\norm{g_k-f}_{L^p(X,\theta)}\leq\|g_k-\tilde{f}_k\|_{L^p(X,\theta)}
+\|\tilde{f}_k-f_k\|_{L^p(X,\theta)}+\norm{f_k-f}_{L^p(X,\theta)}.
\]
Observe that $f_k$ converges pointwise $\theta$-a.e. to $f$ and $\abs{f_k}\leq\abs{f}$, then by
the dominated convergence theorem we get $\lim_{k\ra\infty}\norm{f_k-f}_{L^p(X,\theta)}=0$.
Furthermore
\begin{gather*}
\lim_{k\ra\infty}\|g_k-\tilde{f}_k\|_{L^p(X,\theta)}\leq
\pa{\theta(X)}^{\frac{1}{p}}\lim_{k\ra\infty}\|g_k-\tilde{f}_k\|_\infty=0,
\end{gather*}
and
\begin{align*}
\lim_{k\ra\infty}\|\tilde{f}_k-f_k\|_{L^p(X,\theta)} &=
\lim_{k\ra\infty}\pa{\int_{\set{\tilde{f}_k\neq f_k}}\abs{\tilde{f}_k(x)-f_k(x)}^pd\theta(x)}^{\frac{1}{p}}
\\
&\leq 2^{\frac{p-1}{p}}(2^p+1)^{\frac{1}{p}} \lim_{k\ra\infty} k
\pa{\theta\pa{\set{x\in X\tc \tilde{f}_k(x)\neq f_k(x)}}}^{\frac{1}{p}}\\
&\leq 2^{\frac{p-1}{p}}(2^p+1)^{\frac{1}{p}}\lim_{k\ra\infty}\frac{k}{2^{k/p}}=0.\qedhere
\end{align*}
\end{proof}

In the following proposition we state an approximation result for Lipschitz continuous and bounded functions by means of cylindrical regular functions. We give just a sketch of the proof of the result emphasizing the construction of the approximant sequence, see  \cite[Section 2.1]{DPT01} and the proof of \cite[Lemma 2.5]{PES-ZA1} for further details.

\begin{prop}\label{Task}
Assume that Hypothesis \ref{hyp_App} holds true and let $g\in {\rm Lip}_b(X)$, with ${\rm Lip}\,g \leq L$.
Then, there exists a sequence $\{g_{m,n}\,|\,m,n\in\N\}\subseteq \mathcal{F}C_A^2(X)$ such that
\begin{align*}
\lim_{n\ra+\infty}\lim_{m\ra+\infty}g_{m,n}(x)=g(x),\qquad x\in X;
\end{align*}
and
\begin{align*}
\sup_{m,n\in\N}\norm{g_{m,n}}_\infty\leq\norm{g}_\infty,\qquad \sup_{m,n\in\N}\norm{Dg_{m,n}}_\infty\leq L.
\end{align*}
\end{prop}
\begin{proof}
Let $\{h_k\,|\,k\in\N\}$ be the orthonormal basis fixed in Hypothesis \ref{hyp_App}. For every
$n\in\N$ consider the function $\psi_n:\R^n\ra\R$ defined as
\begin{align*}
\psi_n(\xi):=g\pa{\sum_{k=1}^n\xi_kh_k},\qquad \xi=(\xi_1,\ldots,\xi_n)\in\R^n.
\end{align*}
Let $\rho\in C_b^\infty(\R^n)$ with support contained in the unit ball and such that
$\int_{\R^n}\rho(\eta)d\eta=1$. For every $m\in\N$ consider
\begin{align*}
\psi_{m,n}(\xi):=\int_{\R^n}\psi_{n}\pa{\xi-m^{-1}\eta}\rho(\eta)d\eta,\qquad
\xi=(\xi_1,\ldots,\xi_n)\in\R^n.
\end{align*}
Letting $g_{m,n}(x):=\psi_{m,n}(\langle x,h_1\rangle,\ldots, \langle x,h_n\rangle)$, the thesis
follows by standard arguments as in \cite[Section 2.1]{DPT01} and the proof of \cite[Lemma 2.5]{PES-ZA1}.
\end{proof}

\end{document}